%% file: main.tex
\documentclass[reqno, 11pt, american]{amsart}
\title[Quantum Limits of perturbations of $\Lap$ by $\delta_{\gamma}$ on $\mathbb{S}^{2}$]{Quantum Limits of the Laplacian perturbed along a geodesic on $\mathbb{S}^{2}$}

\author[S. Verdasco]{Santiago Verdasco}
\address{M$^2$ASAI. Universidad Politécnica de Madrid, ETSI Navales, Avda. de la Memoria, 4, 28040, Madrid, Spain.}
\email{santiago.verdasco@upm.es}

\input{Repository/packages}

\input{Repository/commands}

\def \sqLap{\sqrt{\Lap}}

\NewDocumentCommand{\sympf}{O{} m m O{}}{{#1\{ #2 \, , #3 #1\}}_{#4}}

\DeclareMathOperator{\Speceven}{Spec_{even}}
\DeclareMathOperator{\Specodd}{Spec_{odd}}

\begin{document}

\begin{abstract}
This article studies the high-frequency behavior of eigenstates of perturbations of the Laplace-Beltrami operator on the two-sphere $\S[2]$ by a measure supported on an equator. We are interested in understanding to what extent this behavior can be described in terms of the geodesic flow of the sphere. This is done by analyzing quantum limits and semiclassical measures of sequences of high-frequency eigenfunctions, which  describe how their $L^2$-masses concentrate in phase space. When the Laplacian on $\S[2]$ is perturbed by a bounded potential, it is known that the family of all possible semiclassical measures is contained in the set of positive measures on the unit cosphere bundle $S^*\S[2]$ that are invariant under geodesic flow (with equality in the unperturbed case). In this article, we show that the presence of a singular delta potential on a closed geodesic results in the existence of sequences of eigenfunctions whose semiclassical measure is not invariant under geodesic flow. In particular, one can find a sequence of eigenfunctions whose energy asymptotically concentrates on the hemisphere bounded by the equator on which the potential is concentrated.
\end{abstract}

\maketitle

\section{Introduction}
\label{Sec:Introduction}

Understanding the properties of eigenfunctions of a self-adjoint Schrödinger operator has been a topic of intensive research from different areas of mathematics (Analysis, Differential Geometry, Number Theory, and Representation theory, among others) over the last sixty years. They arise originally as the basic building blocks for analyzing the dynamics of classical and quantum waves, heat propagation, etc.; see the monograph \cite{Zelditch2017} or the survey \cite{Zelditch2008} for a detailed introduction to the field.

This article focuses on the study of \textit{high-frequency} eigenfunctions (that is, eigenfunctions corresponding to large eigenvalues) of Schrödinger operators on a compact Riemannian manifold $(M, g)$ without boundary. These are defined as follows: let $\Lap$ be the positive Laplace-Beltrami operator on $(M,g)$, and for a real-valued $V \in L^\infty(M)$ form the operator $\Lap {}+ V$. Eigenfunctions of $\Lap {}+ V$ are solutions $u \in L^2(M)$, $u \neq 0$, to the equation
\begin{equation} \label{Eq:SchroOp-EF-equation}
    [\Lap {}+ V] u = \eta \, u \qquad \eta \in \R \ .
\end{equation}
Since $M$ is compact, \eqref{Eq:SchroOp-EF-equation} has a non-zero solution only for an increasing sequence of real numbers $(\eta_{n})_{n \in \N}$, $\eta_n \nearrow + \infty$, called eigenvalues; we write $\Spec(\Lap {}+ V) = \{ \eta_{n} \}_{n \in \N}$, this is the spectrum of $\Lap {}+ V$.

Eigenfunctions encode local and global geometric properties of $(M, g)$: the first eigenfunctions provide information about the local geometry, they depend, among other quantities, on the curvatures; high-frequency eigenfunctions, on the other hand, are related to the global geometry and the long-time dynamics of the geodesic flow. In this article, we are interested in the latter regime; in particular, we are interested in understanding the geometry of those regions on $M$ where sequences of high-frequency eigenfunctions may concentrate. 

A well-established approach to this problem consists in considering a sequence of probability  measures $(\abs{u_n}^2 \D{x})_{n \in \N}$ on $M$ where, for every $n \in \N$, $u_n$ satisfies 
\begin{equation} \label{Eq:HE-EF-equation}
    (\Lap {}+ V) u_{n} = \lambda_{n}^2 u_{n} \ , \qquad \norm{u_n}[L^2(M)] = 1 \ , \qquad \lim_{n \to \infty} \lambda_n^2 \to \infty \ .
\end{equation}
Observe that the sequence of measures $(\abs{u_n}^2 \D{x})_{n \in \N}$ is bounded in $(\Cont(M))^*$, thus, thanks to the Banach-Alaoglu theorem, there exist a positive measure $\nu$ and a subsequence $(\abs{u_{n_j}}^2 \D{x})_{j \in \N}$ such that
\[
\lim_{j \to \infty} \int_{M} a(x) \abs{u_{n_j}(x)}^2 \D{x} = \int_{M} a(x) \nu(\D{x}) \qquad \forall \, a \in \Cont(M) \ .
\]
Measures $\nu$ arising this way are called \emph{quantum limits}. It is known that the structure of quantum limits is related to the dynamics of the geodesic flow of $(M,g)$ which is naturally defined as a Hamiltonian flow on the phase space $T^*M$, the cotangent bundle of $M$. It is therefore natural to consider lifts of the probability measures $\abs{u_{n_j}(x)}^2 \D{x}$ to phase space, and analyze their interaction with the geodesic flow. One way to do this is to considering (semiclassical) \textit{Wigner distributions}, which are defined for $u \in L^2(M)$ and $h > 0$ by
\begin{equation} \label{Eq:def-Wigner_dist}
    W_{h}[u] (a) \coloneqq \ip{u}{\Op[h]{a} u}[L^2(M)] \ , \qquad a \in \CinfK(T^*M) \ ,
\end{equation}
where $\Op[h]{a}$ is the semiclassical pseudodifferential operator with symbol $a$ obtained by the semiclassical Weyl Quantization; see \cite[Chapter 14]{Zworski2012} for a definition of the operators $\Op[h]{a}$. 

If one chooses a sequence of \emph{semiclassical parameters} $h_n\to0^+$ and replaces $(\abs{u_n}^2 \D{x})_{n \in \N}$ with the sequence of Wigner distributions $(W_{h_n}[u_n])_{n \in \N}$, a weak compactness argument implies the existence of a subsequence $(W_{h_{n_k}}[u_{n_{k}}])_{k \in \N}$ and a positive measure $\mu$ on $T^*M$ such that
\[
\lim_{k \to \infty} \ip{ u_{n_k} }{\Op[h_{n_k}]{a} u_{n_k}}[L^2(M)] = \int_{T^*M} a(x, \xi) \, \mu(\D{x}, \D{\xi}) \qquad \forall \, a \in \CinfK(T^*M) \ .
\]
These measures $\mu$ are called \emph{semiclassical measures}; they simultaneously detect concentration effects \emph{and} oscillations of order $h_n^{-1}$ developed by the sequence $(u_n)_{n \in \N}$. When studying eigenfunctions of a second order differential operator, such as $\Lap {}+ V$, one usually makes the choice $h_n \coloneqq \lambda_{n}^{-1}$,\footnote{In fact, any other sequence $(h_n)_{n \in \N}$ such that $\lim_{n \to \infty} h_n\lambda_{n} = 1$ will produce the same measure $\mu$.} and equation \eqref{Eq:HE-EF-equation} can be rewritten in semiclassical form:
\begin{equation} \label{Eq:EF-equation-semiclassicalPOV}
    [h_n^2 \Lap {}+ h_n^2 V] u_n = u_n \ , \qquad \norm{u_n}[L^2(M)] = 1 \ , \qquad \lim_{n \to \infty} h_n = 0 \ .
\end{equation}
With this choice of $(h_n)_{n\in\N}$, semiclassical measures are always probability measures on $T^*M$. 
If a sequence $(u_n)_{n \in \N}$ satisfying \eqref{Eq:EF-equation-semiclassicalPOV} has a unique quantum limit $\nu$ and a unique semiclassical measure $\mu$ then
\begin{equation} \label{Eq:SDM-QL-relation}
    \int_{T^*M} a(x) \mu(\D{x}, \D{\xi}) = \int_{M} a(x) \nu(\D{x}) \qquad \forall \, a \in \Cont(M) \ .
\end{equation}

Let us denote by $\mathcal{M}_{\mathrm{sc}}(\Lap {}+ V)$ the family of all possible semiclassical measures $\mu$ arising from a sequence of eigenfunctions $(u_n)_{n \in \N}$ as in \eqref{Eq:EF-equation-semiclassicalPOV}. Using the symbolic calculus for pseudodifferential operators, one can show that every $\mu \in \mathcal{M}_{\mathrm{sc}}(\Lap {}+ V)$ belongs to $\mathcal{P}_{\mathrm{inv}}(S^*M)$, the family of all semiclassical measures that enjoy the following two properties:
\begin{enumerate}
    \item $\mu$ is a probability measure, and $\supp \mu \subseteq S^*M$,
    \item $\mu$ is invariant under geodesic flow $\phi_t$ on $T^*M$, that is, $(\phi_t)_{*} \mu = \mu$ for all $t \in \R$.
\end{enumerate}
The problem of characterizing the set $\mathcal{M}_{\mathrm{sc}}(\Lap {}+ V)$ for a given geometry is a very difficult problem, even for $V = 0$; it has been completely solved in very few cases and is the subject of an important open problem, the \textit{Quantum Unique Ergodicity} conjecture by Rudnick and Sarnak \cite{RudnickSarnak1994}.

In compact rank-one symmetric spaces, a setting with a large number of symmetries and periodic geodesic flow, it was proved that $\mathcal{M}_{\mathrm{sc}}(\Lap) = \mathcal{P}_{\mathrm{inv}}(S^*M)$ \cite{JakobsonZelditch1999, AzagraMacia2010, Macia2008}. However, this is a highly unusual situation: many of the known results show that under certain conditions in the geometry or potential, semiclassical measures enjoy additional properties that force $\mathcal{M}_{\mathrm{sc}}(\Lap) \subsetneq \mathcal{P}_{\mathrm{inv}}(S^*M)$. For example, if $V \neq 0$ and either $(M,g)$ is a Zoll manifold \cite{Macia2008, MaciaRiviere2016, MaciaRiviere2019}, or $(M,g)$ has completely integrable geodesic flow \cite{AnantharamanFer-KammererMacia2015, AnantharamanMacia2014, AnantharamanLeautaudMacia2016, MaciaRiviere2018}, it has been shown that any $\mu \in \mathcal{M}_{\mathrm{sc}}(\Lap {}+ V)$ is invariant under an additional Hamiltonian flow on $S^*M$ derived from $V$.

Meanwhile, if $V = 0$ and $(M, g)$ has negative sectional curvature, the results of \cite{Anantharaman2008, AnantharamanNonnenmacher2007Anosovman, Riviere2010} imply that every $\mu \in \mathcal{M}_{\mathrm{sc}}(\Lap)$ has positive entropy, which prevents $\mu$ from being supported inside a closed geodesic. These results are partial advances towards the Quantum Unique Ergodicity Conjecture \cite{RudnickSarnak1994} which states that on a compact closed surface with constant negative curvature, $\mathcal{M}_{\mathrm{sc}}(\Lap) = \{\Liouv\}$, where $\Liouv$ stands for the Liouville measure on $S^*M$ induced by the symplectic structure on $T^*M$. This conjecture was based on the works \cite{Snirelman1974, ColindeVerdiere1985, Zelditch1987}, whose combined efforts led to what is currently known as the Quantum Ergodicity Theorem: given a closed Riemannian manifold $(M,g)$ with ergodic geodesic flow and an orthonormal basis of eigenfunctions of $\Lap$, there exists a density-one subsequence whose single semiclassical measure is $\Liouv$.

In \cite{Seba1990}, S{\v e}ba introduced a different model he conjectured to possess quantum chaotic features. This model is very different from surfaces of negative curvature, as it is obtained by perturbing any Laplacian $\Lap$ by Dirac delta potentials, $\delta_{q}$. What is now known as the S{\v e}ba billiard is obtained as the limiting dynamics of the billiard flow in a rectangle minus a disk of radius $r > 0$ (a Sinai billiard) as $r$ tends to zero.  
The quantizations of the S{\v e}ba billiard can be realized as a one-parameter family of operators, which can be written formally as
\[
\Lap {}+ \alpha \ket{\delta_{q}} \bra{\delta_{q} }
\]
for some constant $\alpha \in \R$ that represents the intensity of the perturbation. Here $\ket{\delta_{q}} \bra{\delta_{q} }$ is short for the finite rank-one projector
\begin{equation} \label{Eq:def-deltaq-projector}
    \ket{\delta_{q}} \bra{\delta_{q}} \colon \Cinf(M) \to \Dist(M) \ , \qquad \ket{\delta_{q}} \bra{\delta_{q}} u = u(q) \, \delta_{q} \ .
\end{equation}
This projector can be extended continuously to a map from $H^2(M)$ to $H^{-2}(M)$, only when $\dim(M) \leq 3$, which is why these perturbations are considered for this range of dimensions.

The S{\v e}ba billiard motivated the study of point perturbations on tori. The eigenvalue statistics of these perturbed operators have drawn much attention from the mathematical physics community; see \cite{BogomolnyGerlandSchmit2001, FreibergKurlbergRosenzweig2017, Hillairet2002, RahavFishman2002, RudnickUeberschaer2014, Ueberschaer2012} and the references therein. Understanding the properties of the set of semiclassical measures for eigenstates has been a problem of interest in the last 20 years. For example, a Quantum Ergodicity theorem was proved for one point perturbation in the square torus in dimension $2$ \cite{KurlbergUeberschaer2014} and in dimension $3$ \cite{Yesha2015}. However, concentration in the momentum variable for irrational tori has been observed in dimensions $2$ and $3$ \cite{KurlbergUeberschaer2017, KurlbergRosenzweig2017, KurlbergLesterRosenzweig2023}. These results depend on fine spectral estimates coming from Analytic Number Theory, exploiting the lattice structure underlying the spectrum of $\Lap$ in tori.

In general, point perturbations of $\Lap$ on a finite set $Q$ may be considered; these are formally represented as
\begin{equation} \label{Eq:formal-Q-point-pert}
    \Lap {}+ \sum_{p, q \in Q} a_{pq} \ket{\delta_{q}} \bra{\delta_{p}}
\end{equation}
for a Hermitian matrix $\mathbb{A} = (a_{pq})_{p, q \in Q}$. The non-diagonal entries of the matrix $\mathbb{A}$ represent non-local interactions between the points of $Q$. In the semiclassical regime, it makes sense to make the perturbation parameter $\mathbb{A}$ depend on $h_n$; see \cite{Ueberschaer2014, Shigehara1993} for a physical heuristic supporting this claim. In this regime, in the spheres $\S[2]$ and $\S[3]$, the author proved the existence of sequences of eigenfunctions of point-perturbations of $\Lap$ on $Q = \{-q, q\}$ whose semiclassical measure is not invariant under geodesic flow \cite{Verdasco2026spheres}. This phenomenon is specific to settings where the subset $Q$ can "trap" part of the geodesic flow. More precisely, as soon as the set $Q$ is \emph{non-self-focal}, that is, for every $q \in Q$ the set of unit-speed co-vectors for which the geodesic flow returns to $Q$ in finite time has zero measure on $S_{q}^*M$, then every semiclassical measure is invariant under the geodesic flow \cite{Verdasco2026}.

In this article, we study singular perturbations of $\Lap$ on $\S[2]$ by the measure $\delta_{\gamma}$, where $\gamma$ is a closed geodesic on $\S[2]$ and
\begin{equation}
    \delta_{\gamma}(a) \coloneqq \int_{0}^{2\pi} a(\gamma(s)) \frac{\D{s}}{2\pi} \ , \qquad  a \in \Cont(\S[2]) \ .
\end{equation}
These are operators that formally look like (see Section \ref{Sec:Singular_perturbation_of_Laplacian} for a precise definition)
\begin{equation} \label{Eq:formal-pert-along-gamma}
    \Lap {}+ \alpha \ket{\delta_{\gamma}} \bra{\delta_{\gamma}} \ , \qquad \alpha \in \R \ ,
\end{equation}
where $\ket{\delta_{\gamma}} \bra{\delta_{\gamma}}$ is defined analogously to \eqref{Eq:def-deltaq-projector}. In particular, we are interested in the concentration and high-oscillation properties of sequences of high-energy eigenfunctions of operators of the form $\Lap {}+ \alpha \ket{\delta_{\gamma}} \bra{\delta_{\gamma}}$.

Given a sequence of singular perturbations of $\Lap$ on $\S[2]$ by $\delta_{\gamma}$, $(\Lap {}+ \alpha_n \ket{\delta_{\gamma}} \bra{\delta_{\gamma}})_{n \in \N}$, for every $n \in \N$ let $u_n \in L^2(\S[2])$ and $h_n > 0$ be such that
\begin{equation} \label{Eq:semiclassical-EFeq-Lapgamma}
    h_n^2\big[ \Lap {}+ \alpha_n \ket{\delta_{\gamma}} \bra{\delta_{\gamma}} \big] u_n = u_n \ , \qquad \norm{u_n}[L^2(\S[2])] = 1 \ , \qquad \lim_{n \to \infty} h_n = 0 \ .
\end{equation}
Let us denote
\[
\mathcal{M}_{\mathrm{sc}} \Big( \big(\Lap {}+ \alpha_n \ket{\delta_{\gamma}} \bra{\delta_{\gamma}} \big)_{n \in \N} \Big)
\]
as the set of all possible weak-$\star$ accumulation points of the sequence of Wigner distributions $(W_{h_n}[u_n])_{n \in \N}$ where $u_n$ and $h_n$ satisfy \eqref{Eq:semiclassical-EFeq-Lapgamma}. The main results in this article state that $\mathcal{M}_{\mathrm{sc}} ( (\Lap {}+ \alpha_n \ket{\delta_{\gamma}} \bra{\delta_{\gamma}} )_{n \in \N} )$ always contains the set of invariant probability measures supported on $S^*M$ and that, in some situations, the inclusion is strict.

\begin{Theorem} \label{Thm:maintheorem-invariantSDM}
    Let $\gamma$ be a closed geodesic in $\S[2]$, and fix a sequence $(\Lap {}+ \alpha_n \ket{\delta_{\gamma}} \bra{\delta_{\gamma}})_{n \in \N}$ of singular perturbations of $\Lap$ by the distribution $\delta_{\gamma}$. Then
    \begin{equation} \label{Eq:contenido-sets-invariance}
    \mathcal{P}_{\mathrm{inv}}(S^*\S[2]) \subseteq \mathcal{M}_{\mathrm{sc}} \Big( \big( \Lap {}+ \alpha_n \ket{\delta_{\gamma}} \bra{\delta_{\gamma}} \big)_{n \in \N} \Big) \ .
    \end{equation}
    Moreover, there exist sequences $(\Lap {}+ \alpha_n \ket{\delta_{\gamma}} \bra{\delta_{\gamma}})_{n \in \N}$ for which
    \begin{equation} \label{Eq:equality-sets-invariance}
        \mathcal{P}_{\mathrm{inv}}(S^*\S[2]) = \mathcal{M}_{\mathrm{sc}} \Big( \big(\Lap {}+ \alpha_n \ket{\delta_{\gamma}} \bra{\delta_{\gamma}} \big)_{n \in \N} \Big) \ .
    \end{equation}
\end{Theorem}

\begin{Remark}
    In fact, sequences $(\Lap {}+ \alpha_n \ket{\delta_{\gamma}} \bra{\delta_{\gamma}})_{n \in \N}$ such that \eqref{Eq:equality-sets-invariance} holds are abundant as described in more detail in Remark \ref{Rmk:genericity_condition-newEV_cap_oldEV=empty}, after Theorem \ref{Thm:equal-set_SDM}.
\end{Remark}

\begin{Theorem} \label{Thm:maintheorem-QuantumLimit}
    Let $\gamma$ be a closed geodesic in $\S[2]$ and let $\Omega_{+}$ and $\Omega_{-}$ be the two connected components of $\S[2] \setminus \gamma$. Consider the probability measure on $\S[2]$,
    \[
    \nu_\gamma (\D{x}) \coloneqq \frac{2}{\pi} \frac{1}{\cos(d(x, \gamma))} \frac{\D{x}}{\vol(\S[2])} \ .
    \]
    There exists a sequence $(\Lap {}+ \alpha_n \ket{\delta_{\gamma}} \bra{\delta_{\gamma}})_{n \in \N}$ such that for every $m \in [0,2]$ there exist $(u_n)_{n \in \N} \subseteq L^2(\S[2])$ and $(h_n)_{n \in \N} \subseteq (0,1)$ satisfying \eqref{Eq:semiclassical-EFeq-Lapgamma} such that for all $b \in \Cont(\S[2])$,
    \begin{equation} \label{Eq:singularQL}
        \lim_{n \to \infty} \int_{\S[2]} b(x) \abs{u_n(x)}^2 \D{x} = \int_{\S[2]} b(x) \, \big[ m \charf_{\Omega_{+}}(x) + (2-m) \charf_{\Omega_{-}} (x) \big]\nu_{\gamma}(\D{x}) \ .
    \end{equation}
    For $m \neq 1$, these quantum limits cannot be the projection onto the base manifold $\S[2]$ of a measure on $S^*\S[2]$ that is invariant under geodesic flow.
\end{Theorem}

The non-invariance can be understood as an outcome of a reflection phenomenon on the hypersurface $\gamma$ in the semiclassical limit. Reflection phenomena due to codimension-one perturbations of a Schrödinger operator in the semiclassical limit ($h \to 0^+$) have been  studied in the literature, mostly for Schrödinger operators $H_{h} \coloneqq h^2\Lap {}+ V$ with potentials having a finite-order conical singularity on some embedded hypersurface $Y \subseteq M$: $V \in \Cinf(M \setminus Y; \R)$, but $V \in \Cont^{k}(M)$ for some $k \in \N$. If $M = \R$ is equipped with the canonical metric, \cite{Berry1982} describes a connection through the WKB method between the regularity of $V$ on $Y = \{y_0\}$ and the size of the reflection term. More recently, \cite{FermanianKammererGerardLasser2013, Chabu2017, GannotWunsch2023} extended these results to higher dimensions and curved geometries. In \cite{GalkowskiWunsch2024, BurqDehmanLeRousseau2024}, invariance under classical dynamics was studied for semiclassical measures arising from sequences of eigenfunctions of $H_{h}$ in the semiclassical limit. Reflection phenomena also contribute to the spectral properties of these operators in that limit; see, for instance, \cite{WunschYangZou2025, Zou2025}. The perturbations analyzed in this article can be interpreted as a limiting case for these type of potentials.

\subsection{Outline of the proof}

Proving Theorems \ref{Thm:maintheorem-QuantumLimit} and \ref{Thm:maintheorem-invariantSDM} requires precise knowledge about the eigenfunctions of singular perturbations of $\Lap$ by $\delta_{\gamma}$. Section \ref{Sec:Singular_perturbation_of_Laplacian} is devoted to proving that these operators have compact resolvent, characterizing their spectrum, and computing their eigenspaces. Eigenfunctions $u$ can be classified into three categories: (a) old eigenfunctions, which are eigenfunctions of $\Lap$ with the added condition $\delta_{\gamma} u = 0$; (b) new eigenfunctions, which have a singularity along $\gamma$; and (c) linear combinations of old eigenfunctions and new eigenfunctions.

The family of old eigenfunctions is rich enough in the high-energy regime to produce any invariant probability measure; an application of \cite[Theorem 1.3]{Verdasco2026spheres} gives \eqref{Eq:contenido-sets-invariance}. On the other hand, new eigenfunctions have a very rigid structure: they are infinite linear combinations of zonal harmonics. First, in Section \ref{Sec:QuasimodesforGF} we show that a proper truncation of the series produces a quasimode approximation. In the proof of Theorem \ref{Thm:SDM-newEF_is_muq}, we use this quasimode approximation, ladder operators for zonal harmonics, which are $h$-$\Psi$DO, together with symbolic calculus to eventually approximate a new eigenfunction by a $h$-$\Psi$DO acting on a single zonal harmonic $z_{\ell_n}$, obtaining an asymptotic of the form
\[
\ip*{u_n}{\Op[h_n]{a} u_n }[L^2(\S[2])] = \ip*{z_{\ell_n}}{\Op[h_n]{a \abs*{\Gamma^{\Upsilon}}^2} z_{\ell_n} }[L^2(\S[2])] + \littleo(1_{n}) + \Upsilon^{-\frac{1}{2}} \BigO(1) \qquad \text{as $n \to \infty$} \ ,
\]
for a certain symbol $\Gamma^{\Upsilon}$ and a large parameter $\Upsilon \in \N$ that comes from the truncation procedure. Taking the limit $n \to \infty$ first, and then analyzing $\abs{\Gamma^{\Upsilon}}^2$, we find that $\abs{\Gamma^{\Upsilon}}^2 \to 1$ as $\Upsilon \to \infty$, thus any semiclassical measure arising from a sequence of high-energy new eigenfunctions is invariant under the geodesic flow. For most singular perturbations of $\Lap$ by $\delta_{\gamma}$, old and new eigenfunctions are the only eigenfunctions of the perturbed operator, thus \eqref{Eq:equality-sets-invariance} holds.

In this setting, a non-invariant semiclassical measure can arise only from a sequence of linear combinations of old and new eigenfunctions. To show that, in the proof of Theorem \ref{Thm:end-theorem-non-invariance}, we consider a sequence of linear combinations of zonal harmonics and new eigenfunctions
\[
u_n \coloneqq \varphi \, z_{\ell_n} + \psi \, g_{n} \ ,
\]
and using the approximation described above, one obtains an asymptotic of the form
\[
\ip*{u_n}{\Op[h_n]{a} u_n }[L^2(\S[2])] = \ip*{z_{\ell_n}}{\Op[h_n]{a \abs*{\varphi + \psi \Gamma^{\Upsilon}}^2} z_{\ell_n} }[L^2(\S[2])] + \littleo(1_{n}) + \Upsilon^{-\frac{1}{2}} \BigO(1) \qquad \text{as $n \to \infty$} \ .
\]
This time, after taking $n \to \infty$, the symbol $\abs{\varphi + \psi \Gamma^{\Upsilon}}^2$ converges to a non-constant function as $\Upsilon \to \infty$, hence obtaining a non-invariant semiclassical measure.

\subsection*{Acknowledgments}

The author is indebted to Fabricio Macià as this problem was suggested by him. This research has been supported by grants PID2021-124195NB-C31 and PID2024-158664NB-C21 from Agencia Estatal de Investigación (Spain), and by grant VPREDUPM22 from Programa Propio UPM.


\section{Singular perturbations of the Laplacian}
\label{Sec:Singular_perturbation_of_Laplacian}

\subsection{Description of singular perturbations of \texorpdfstring{$\Lap$}{Laplacian} by \texorpdfstring{$\delta_{\gamma}$}{delta-gamma}}
\label{subSec:Singular_perturbation-Description}

Let $\Lap$ denote the positive Laplace-Beltrami operator on $(\S[2], \mathrm{can.})$, and let $\D{x}$ denote the Riemannian measure on $\S[2]$. Consider the Hilbert space of complex-valued functions $L^2(\S[2], \D{x})$ with the inner product
\[
\ip{u}{v}[L^2(\S[2])] \coloneqq \int_{\S[2]} \conj{u(x)} v(x) \D{x} \ .
\]
As an operator on $L^2(\S[2])$, let $\Lap$ be defined on $\dom(\Lap) \coloneqq H^2(\S[2])$; then $\Lap$ is a self-adjoint operator on $L^2(\S[2])$.

Let $\gamma$ be a closed geodesic on $\S[2]$ and define the distribution
\begin{equation} \label{Eq:def-deltagamma}
    \delta_{\gamma} u \coloneqq \int_{0}^{2\pi} u(\gamma(s)) \frac{\D{s}}{2\pi} \ , \qquad u \in \Cinf(\S[2]) \ .
\end{equation}
This distribution can be continuously extended to $\dom(\Lap)$. Here, we are interested in operators that formally look like
\begin{equation} \label{Eq:Lap+deltagamma-formal}
    \Lap {}+ \alpha \ket{\delta_{\gamma}} \bra{\delta_{\gamma}}
\end{equation}
for some $\alpha \in \R$. These formal operators are considered singular perturbations of $\Lap$ because its natural domain is no longer $\dom(\Lap)$, but a different one; for every $u \in \dom(\Lap) = H^2(\S[2])$,
\[
\big[ \Lap {}+ \alpha \ket{\delta_{\gamma}} \bra{\delta_{\gamma}} \big] u = \Lap u + \alpha [\delta_{\gamma} u] \delta_{\gamma} \ ,
\]
which does not belong to $L^2(\S[2])$ as soon as $\delta_{\gamma} u \neq 0$.

These singular perturbations are defined via the von Neumann theory of symmetric extensions (see Appendix \ref{App:vonNeumann_theory}): they are self-adjoint extensions of the operator $T \coloneqq \Lap|_{\dom(T)}$ \footnote{In general, a symmetric operator may have no self-adjoint extensions, such as the operator $\frac{1}{i}\partial_{x}$ defined on $\CinfK(0, \infty) \subseteq L^2(0, \infty)$. See \cite[Lemma 10.1]{Nelson1959} for a geometric explanation of this example.}, where
\begin{equation} \label{Eq:def-op_T}
    \dom(T) \coloneqq \big\{ u \in \dom(\Lap) \colon \, \delta_{\gamma} u = 0 \big\} \ . 
\end{equation}

\begin{Prop} \label{Prop:family-perturbation-Lap-gamma}
    Let $\gamma$ be a closed geodesics in $\S[2]$, and let $\Omega_{+}$ and $\Omega_{-}$ be the two connected components of $\S[2] \setminus \gamma$. Define the following two functions on $L^2(\S[2])$
    \begin{equation} \label{Eq:functions-ugamma01}
        u_{\gamma 0}(x) \coloneqq 1 \ , \quad u_{\gamma 1}(x) \coloneqq - \tfrac{1}{2} \sin \big( d(x, \gamma) \big) \ , \qquad x \in \S[2] \ .
    \end{equation}
    The self-adjoint extensions of $T$ are the operators $\{ \Lap_{\gamma, \theta}\}_{\theta \in \Theta}$, $\Theta = \R / (\pi \Z)$, where
    \begin{equation} \label{Eq:dom(Lapgammatheta)}
        \dom(\Lap_{\gamma, \theta}) \coloneqq \dom(T) + \Span[\C]{ \cos \theta \ u_{\gamma 0} + \sin \theta \ u_{\gamma 1} }
    \end{equation}
    and for $u_0 \in \dom(T)$, $c \in \C$,
    \begin{equation} \label{Eq:operator-Lapgammatheta}
        \Lap_{\gamma, \theta} \big[ u_{0} + c (\cos \theta \ u_{\gamma 0} + \sin \theta \ u_{\gamma 1}) \big] \coloneqq \Lap u_{0} + c (2 \sin \theta \ u_{\gamma 1} )\ .
    \end{equation}
    In other words, $\Lap_{\gamma, \theta} \coloneqq T^*|_{\dom(\Lap_{\gamma, \theta})}$. Moreover, $\dom(\Lap_{\gamma, 0}) = \dom(\Lap)$, hence $\Lap_{\gamma, 0} = \Lap$.
\end{Prop}

We leave the proof of Proposition \ref{Prop:family-perturbation-Lap-gamma} to the end of the section, as we need some auxiliary lemmas. 

\begin{Lemma} \label{Lemma:T-closed-symmetric}
    The operator $T$ is a closed symmetric operator in $L^2(\S[2])$.
\end{Lemma}
\begin{proof}
    The distribution $\delta_{\gamma}$ \eqref{Eq:def-deltagamma} is not bounded in $L^2(\S[2])$, therefore $\dom(T)$ is still a dense subspace of $L^2(\S[2])$. Nonetheless, $\delta_{\gamma}$ is bounded in the $H^2(\S[2])$ topology, hence $\dom(T)$ is a closed one-codimensional subspace of $H^2(\S[2]) = \dom(\Lap)$, therefore $T$ is a closed operator. Moreover, since $T$ is the restriction of a symmetric operator, $\Lap$, $T$ is itself symmetric.
\end{proof}

In order to find all self-adjoint extensions of $T$, one starts considering the adjoint operator $T^*$. Recall the definition of $\dom(T^*)$,
\begin{equation} \label{Eq:def-T*op-section2}
    \dom(T^*) \coloneqq \Big\{ u \in L^2(\S[2]) \colon \ \exists \, C_{u} > 0 \ \text{s.t.} \ \abs{\ip{u}{T w}[L^2(\S[2])]} \leq C_{u} \norm{w}[L^2(\S[2])] \quad \forall \, w \in \dom(T) \, \Big\} \ ,
\end{equation}
and the definition of $T^*u$ via the Riesz representation theorem: for $u \in \dom(T^*)$, $T^*u \in L^2(\S[2])$ is the unique vector such that
\[
\ip{T^*u}{w}[L^2(\S[2])] = \ip{u}{T w}[L^2(\S[2])] \qquad \forall \, w \in \dom(T) \ .
\]

\begin{Prop} \label{Prop:computation-dom(T*)}
    Let $\gamma$ be a closed geodesic in $\S[2]$ and consider $u_{\gamma 0}$, $u_{\gamma 1}$ as defined in \eqref{Eq:functions-ugamma01}. If $T$ is defined as in \eqref{Eq:def-op_T}, then $\dim(\dom(T^*)/ \dom(T)) = 2$ and
    \[
    \dom(T^*) = \dom(T) + \Span[\C]{u_{\gamma 0}, u_{\gamma 1}} \ ,
    \]
    and
    \begin{equation} \label{Eq:T*ugamma01}
        T^* u_{\gamma 0} = 0 \ , \qquad T^* u_{\gamma 1} = 2 u_{\gamma 1} \ .
    \end{equation}
\end{Prop}

\begin{proof}
    Thanks to Lemma \ref{Lemma:T-closed-symmetric} we know that $T$ is a closed symmetric operator. Observe that $\Lap$ is a self-adjoint extension of $T$ and that $\dim( \dom(\Lap) / \dom(T) ) = 1$ because $\dom(T)$ is the kernel of a linear functional on $\dom(\Lap)$, thus Lemma \ref{Lemma:def-J}(3) implies $\dim (\dom(T^*) / \dom(T) ) = 2$. It is clear that
    \begin{equation} \label{Eq:dom(Lap)=dom(T)+Span(ugamma0)}
        \dom(\Lap) = \dom(T) + \Span[\C]{u_{\gamma 0}} \ , \qquad \text{and} \quad u_{\gamma 1} \notin \dom(\Lap) \ .
    \end{equation}
    In addition, using that
    \[
    [\Lap u_{\gamma 1}] (x) = 2 u_{\gamma 1} (x) \qquad \forall \, x \in \Omega_{+} \cup \Omega_{-} \ ,
    \]
    and integration by parts on $\Omega_{+}$ and $\Omega_{-}$, we find that for all $w \in H^2(\S[2])$,
    \[
    \ip{u_{\gamma 1}}{\Lap w}[L^2(\S[2])] = \delta_{\gamma} w + \ip{2 u_{\gamma 1}}{w}[L^2(\S[2])] \ .
    \]
    Therefore, $u_{\gamma 1} \in \dom(T^*)$ and $T^* u_{\gamma 1} = 2 u_{\gamma 1}$.
\end{proof}

\begin{proof}[Proof of Proposition \texorpdfstring{\ref{Prop:family-perturbation-Lap-gamma}}{perturbaciones Lap gamma}]
    Let $A$ be a self-adjoint extension of $T$. Since $A$ is self-adjoint, then $A = A^* \subset T^*$, thus $A = T^*|_{\dom(A)}$ for some subspace $\dom(T) \subseteq \dom(A) \subseteq \dom(T^*)$. Thanks to Proposition \ref{Prop:self-adjoint-characterization}, we know that $\dom(A)$ must be a Lagrangian subspace for the skew-Hermitian form
    \begin{equation} \label{Eq:def-omega-sec2}
        \omega(u, v) \coloneqq \ip{u}{T^*v} - \ip{T^*u}{v} \ , \qquad u, v \in \dom(T^*) \ .
    \end{equation}
    Thanks to Proposition \ref{Prop:computation-dom(T*)}, we know that
    \[
    \dom(T^*) = \dom(T) + \Span[\C]{u_{\gamma 0}, u_{\gamma 1}} \ ,
    \]
    thus
    \[
    \dom(A) = \dom(T) + \Span[\C]{v}
    \]
    for some $v \in \Span[\C]{u_{\gamma 0}, u_{\gamma 1}}$, due to Lemma \ref{Lemma:def-J}(3). Thanks to Lemma \ref{Lemma:characterization_dom(T)-via_omega}, it is necessary and sufficient for $\dom(A)$ to be Lagrangian that $\omega(v, v) = 0$.

    Note that
    \begin{equation} \label{Eq:omega(ugamma01)}
        \begin{aligned}
            \omega(u_{\gamma 0}, u_{\gamma 0}) & = 0 \\
            \omega(u_{\gamma 1}, u_{\gamma 1}) & = 0 \\
            \omega(u_{\gamma 1}, u_{\gamma 0}) & = 0 - \ip{2 u_{\gamma 1}}{ u_{\gamma 0}} = \int_{\Omega_{+}} \sin(d(x, \gamma)) \D{x} = 2\pi
        \end{aligned}
    \end{equation}
    therefore, if we write $v = a u_{\gamma 0} + b u_{\gamma 1}$ for $a,b \in \C$, we see that
    \[
    \omega(v,v) = 0 \quad \iff \quad \conj{b} a - \conj{a} b = 0 \quad \iff \quad \Im(\conj{b} a) = 0 \ .
    \]
    Without loss of generality, we may assume that $a \in \R$ and that $\abs{a}^2 + \abs{b}^2 = 1$, hence $b \in \R$ and the only solutions are
    \[
    \begin{dcases}
        a = \cos \theta \\ 
        b = \sin \theta
    \end{dcases}
    \ , \qquad \theta \in \R \ .
    \]
    This gives that $\dom(A)$ must be of the form \eqref{Eq:dom(Lapgammatheta)}. \eqref{Eq:operator-Lapgammatheta} is a direct consequence of \eqref{Eq:T*ugamma01} and $A = T^*|_{\dom(A)}$.

    The fact that setting $\theta = 0$ lets you recover the Laplacian $\Lap$ is equivalent to the identity
    \[
    \dom(\Lap) = \dom(T) + \Span[\C]{u_{\gamma 0}} = \dom(\Lap_{\gamma, 0}) \ . \qedhere
    \]
\end{proof}

\subsection{Spectrum of singular perturbations of \texorpdfstring{$\Lap$}{Laplacian} by \texorpdfstring{$\delta_{\gamma}$}{delta-gamma}}
\label{subSec:Spectrum-Lapgammatheta}

Thanks to the Rellich-Kondrachov theorem, the embedding $H^2(\S[2]) \hookrightarrow L^2(\S[2])$ is compact; thus $\Lap$ has compact resolvent and
\[
\Spec(\sqLap) = \{ \lambda_{\ell} \}_{\ell \in \N} \subseteq [0, \infty) \ , \qquad \lambda_{\ell}^2 = \ell(\ell + 1) \ .
\]
In this section we will prove that operators $\Lap_{\gamma, \theta}$ have compact resolvent, among other of their spectral properties.

For $z \in \C \setminus \Spec(\Lap)$, let $G_{z}^{\gamma} \in L^2(\S[2])$ be such that
\begin{equation} \label{Eq:def-Gzgamma}
    \ip{G_{z}^{\gamma}}{ (\Lap {}- \conj{z}) u}[L^2(\S[2])] = \delta_{\gamma} u 
\end{equation}
for all $u \in \dom(\Lap) = H^2(\S[2])$. The existence and uniqueness of the function $G_{z}^{\gamma}$ for $z \in \C \setminus \Spec(\Lap)$ is granted by the Riesz representation theorem: for every $z \in \C \setminus \Spec(\Lap)$, the functional
\[
F_{\gamma, z} \colon L^2(\S[2]) \to \C \ , \qquad F_{\gamma, z}w = \delta_{\gamma} \big[ (\Lap {}- \conj{z})^{-1} w \big]
\]
is bounded in $L^2(\S[2])$, hence there exists a unique $G_{z}^{\gamma} \in L^2(\S[2])$ such that
\[
F_{\gamma, z}w = \ip{G_{z}^{\gamma}}{w}[L^2(\S[2])] \qquad \forall \, w \in L^2(\S[2]) \ .
\]
In order to introduce the main result of this section, Theorem \ref{Thm:Spectrum-Lapgammatheta}, we need to state first a technical proposition that extends the definition of $G_{z}^{\gamma}$ from $\C \setminus \Spec(\Lap)$ to $\C \setminus \Speceven(\Lap)$, where
\begin{equation} \label{Eq:def-Speceven(Lap)}
    \Speceven(\Lap) \coloneqq \{ \lambda_{\ell}^2 \colon \ \text{$\ell \in \N$ even} \} \ .
\end{equation}

\begin{Prop} \label{Prop:Existence-Uniqueness-Gzgamma}
    Let $\gamma$ be a closed geodesic in $\S[2]$. The holomorphic map
    \begin{equation} \label{Eq:def-holomap-calGgamma}
        \mathcal{G}^{\gamma} \colon \C \setminus \Spec(\Lap) \to L^2(\S[2]) \ , \qquad \mathcal{G}^{\gamma}(z) = G_{z}^{\gamma} \ ,
    \end{equation}
    admits a (unique) holomorphic extension to $\C \setminus \Speceven(\Lap)$ such that
    \begin{equation} \label{Eq:def-Gzgamma-extension}
        \ip{G_{z}^{\gamma} }{ (\Lap {}- \conj{z}) u}[L^2(\S[2])] = \delta_{\gamma} u \qquad \forall \, u \in \dom(\Lap) 
    \end{equation}
    holds for all $z \in \C \setminus \Speceven(\Lap)$\footnote{We make the little abuse of notation by writing $G_{z}^{\gamma}$ instead of $\mathcal{G}^{\gamma}(z)$.}. Consequently
    \begin{equation} \label{Eq:T*Gzgamma}
        T^* G_{z}^{\gamma} = z \, G_{z}^{\gamma} \qquad \forall \, z \in \C \setminus \Speceven(\Lap) \ .
    \end{equation}
    See \eqref{Eq:EF-expansion-Gzgamma} for an eigenfunction expansion of $G_{z}^{\gamma}$.
\end{Prop}

\begin{Theorem} \label{Thm:Spectrum-Lapgammatheta}
    Let $\gamma$ be a closed geodesic in $\S[2]$, and for $\theta \in \Theta = \R / (\pi \Z)$, $\theta \neq 0$, let $\Lap_{\gamma, \theta}$ be a self-adjoint perturbation of $\Lap$ on $\S[2]$ by $\delta_{\gamma}$ ($\Lap_{\gamma, \theta} \neq \Lap$, see Proposition \ref{Prop:family-perturbation-Lap-gamma}). For every $z \in \C \setminus \big( \Spec(\Lap) \cup \Spec(\Lap_{\gamma, \theta}) \big)$, there exists $\alpha(\theta, z) \in \C$ such that
    \begin{equation} \label{Eq:resolvent-identity-Lapgammatheta}
        \Big[ (\Lap_{\gamma, \theta} {}- z)^{-1} - (\Lap {}- z)^{-1} \Big] f = \Big[ \alpha(\theta, z) \ip*{G_{\conj{z}}^{\gamma}}{f}[L^2(\S[2])] \Big] G_{z}^{\gamma} \qquad \forall \, f \in L^2(\S[2]) \ ,
    \end{equation}
    thus the resolvent of $\Lap_{\gamma, \theta}$ is compact. Moreover, for $\eta \in \R \setminus \Speceven(\Lap)$ we have
    \[
    \ker(\Lap_{\gamma, \theta} {}- \eta) = \begin{dcases}
        \ker(\Lap {}- \eta) + \Span[\C]{ G_{\eta}^{\gamma} } & \quad \text{if $\eta$ is a solution to \eqref{Eq:equation_eta-newEV},} \\
        \ker(\Lap {}- \eta) & \quad \text{otherwise,}
    \end{dcases}
    \]
    where (see Lemma \ref{Lemma:Gzgamma-ugamma1-dom(Lapgammatheta)} below)
    \begin{equation} \label{Eq:equation_eta-newEV}
        \delta_{\gamma} \Big( G_{\eta}^{\gamma} - \frac{1}{2\pi} u_{\gamma 1} \Big) = \frac{1}{2\pi} \cotan \theta \ .
    \end{equation}
    Meanwhile, for every $\ell \in \N$,
    \[
    \ker(\Lap_{\gamma, \theta} {}- \lambda_{2\ell}^{2} ) = \{ u \in \ker(\Lap {}- \lambda_{2\ell}^{2}) \colon \delta_{\gamma} u = 0 \} \ .
    \]
\end{Theorem}

\begin{Remark} \label{Rmk:EV-equation}
    For every $\ell \geq 1$ and $\theta \in (0, \pi)$, equation \eqref{Eq:equation_eta-newEV} has a unique solution $\eta_{\ell, \theta} \in (\lambda_{2\ell}^2, \lambda_{2\ell + 2}^2)$. See Lemma \ref{Lemma:EV-equation} at the end of the section.
\end{Remark}

We postpone the proof of Theorem \ref{Thm:Spectrum-Lapgammatheta} to the end of the section; we first state and prove some auxiliary results. The first connects $\delta_{\gamma}$ to the eigenspaces of $\Lap$, revealing the reason why $\mathcal{G}_{z}^{\gamma}$ \eqref{Eq:def-holomap-calGgamma} can be extended to $\C \setminus \Speceven(\Lap)$.

For every $\ell \in \N$, let $Y_{\ell}^{\gamma}$ be the unique eigenfunction in $\ker(\sqLap - \lambda_{\ell})$ such that
\begin{equation} \label{Eq:def-Yellgamma}
        \ip{Y_{\ell}^{\gamma}}{u}[L^2(\S[2])] = \delta_{\gamma} u \qquad \forall \, u \in \ker(\sqLap - \lambda_{\ell}) \ .
\end{equation}
Existence and uniqueness of $Y_{\ell}^{\gamma}$ is granted by the Riesz representation theorem in the finite dimensional space $\ker(\sqLap - \lambda_{\ell})$. These eigenfunctions $Y_{\ell}^{\gamma}$ play the same role as zonal harmonics on $\S[2]$ for the distribution $\delta_{\gamma}$ instead of $\delta_{q}$. The $\ell$-th zonal harmonic on $q \in \S[2]$, $Z_{\ell}^{q}$, is defined as the unique eigenfunction $Z_{\ell}^{q} \in \ker(\sqLap - \lambda_{\ell})$ such that
\begin{equation} \label{Eq:def-Zellq}
    \ip{Z_{\ell}^{q}}{u}[L^2(\S[2])] = u(q) \qquad \forall \, u \in \ker(\sqLap {}- \lambda_{\ell}) \ .
\end{equation}
In fact, if $d(\gamma, q) = \frac{\pi}{2}$, eigenfunctions $Z_{\ell}^{q}$ and $Y_{\ell}^{\gamma}$ are related to each other.

\begin{Lemma} \label{Lemma:relation_Yellgamma-Zellq}
    Let $\gamma$ be a closed geodesic in $\S[2]$, and take $q \in \S[2]$ such that $d(q, \gamma) = \frac{\pi}{2}$. For every $\ell \in \N$
    \[
    \begin{dcases}
    Y_{2\ell}^{\gamma} = (-1)^{\ell} \frac{\Gamma(\ell + \frac{1}{2})}{\Gamma(\frac{1}{2}) \Gamma(\ell + 1)} Z_{2\ell}^{q} \ , \\
    Y_{2\ell + 1}^{\gamma} = 0 \ .
    \end{dcases}
    \]
\end{Lemma}

\begin{proof}
    Let $L_{q}$ be the Killing field on $\S[2]$ such that $L_{q}(q) = 0$ and $L_{q}(\gamma(s)) = \gamma'(s)$. The subspace $\Span[\C]{Z_{\ell}^{q}}$ is characterized in $L^2(\S[2])$ by the properties
    \begin{equation} \label{Eq:char-eq-spanZellq}
        [ \sqLap {}- \lambda_\ell ] u = 0 \ , \qquad L_{q} u = 0 \ ,
    \end{equation}
    that is, $u \in \Span[\C]{Z_{\ell}^{q}}$ if and only if $u \in L^2(\S[2])$ satisfies \eqref{Eq:char-eq-spanZellq}. We will show that $Y_{\ell}^{\gamma}$ enjoys \eqref{Eq:char-eq-spanZellq} and then find the appropriate multiplicative constant.

    Fix some $\ell \in \N$. $Y_{\ell}^{\gamma}$ enjoys property (1) by definition. $Y_{\ell}^{\gamma}$ satisfies property (2) due to the fact $[\Lap, L_q] = 0$. If $\Pi_\ell$ is the eigenprojector onto $\ker(\sqLap - \lambda_\ell)$ and $u \in \Cinf(\S[2])$,
    \[
    \ip{L_{q} Y_{\ell}^{\gamma}}{u}[L^2(\S[2])] = -\ip{Y_{\ell}^{\gamma}}{L_{q} u}[L^2(\S[2])] = -\ip{Y_{\ell}^{\gamma}}{ \Pi_\ell L_{q} u}[L^2(\S[2])] = -\ip{Y_{\ell}^{\gamma}}{L_{q} \Pi_\ell u}[L^2(\S[2])] \ .
    \]
    Thus, we may assume without loss of generality that $u \in \ker(\sqLap - \lambda_{\ell})$. Now, for $u \in \ker(\sqLap - \lambda_{\ell})$,
    \[
    \ip{Y_{\ell}^{\gamma}}{L_{q} u}[L^2(\S[2])] = \delta_{\gamma} (L_{q} u) = \int_{0}^{2\pi} [L_{q} u] (\gamma (s)) \frac{\D{s}}{2\pi} = \int_{0}^{2\pi} [u \circ \gamma]'(s) \frac{\D{s}}{2\pi} = 0 \ ,
    \]
    thus $L_{q} Y_{\ell}^{\gamma} = 0$
    
    Since $Y_{\ell}^{\gamma}$ satisfies properties (1) and (2), there exists $\alpha_{\ell} \in \C$ such that 
    \begin{equation} \label{Eq:identity-Yellgamma-Zellq}
        Y_{\ell}^{\gamma} = \alpha_{\ell} Z_{\ell}^{q} \ .
    \end{equation}
    If we know how to compute $\ip{Y_{\ell}^{\gamma}}{Z_{\ell}^{q}}[L^2(\S[2])]$ in two different ways, we will find $\alpha_{\ell}$. For that, we recall that
    \[
    Z_{\ell}^{q} (x) = \frac{2\ell + 1}{\vol(\S[2])} P_{\ell}(\cos r) \ , \qquad r = d(x, q) \ ,
    \]
    where $\{P_{\ell}\}_{\ell \in \N}$ is the family of Legendre polynomials with the normalization $P_{\ell}(1) = 1$ (see \cite[Proposition B.1]{Verdasco2026spheres}, for instance). Note that for every $s \in \R$,
    \[
    Z_{\ell}^{q} (\gamma(s)) = \frac{2\ell + 1}{\vol(\S[2])} P_{\ell}(0) ,
    \]
    therefore
    \[
    \ip{Y_{\ell}^{\gamma}}{Z_{\ell}^{q}}[L^2(\S[2])] = \int_{0}^{2\pi} Z_{\ell}^{q} (\gamma(s)) \frac{\D{s}}{2\pi} = \frac{2\ell + 1}{\vol(\S[2])} P_{\ell}(0) \ .
    \]
    Meanwhile
    \[
    \ip{Y_{\ell}^{\gamma}}{Z_{\ell}^{q}}[L^2(\S[2])] = \conj{\alpha_{\ell}} \ip{Z_{\ell}^{q}}{Z_{\ell}^{q}}[L^2(\S[2])] = \conj{\alpha_{\ell}} \frac{2\ell + 1}{\vol(\S[2])} P_{\ell}(1) \ ,
    \]
    thus
    \begin{equation} \label{Eq:values-alphaell}
        \alpha_{\ell} = P_{\ell}(0) \ .
    \end{equation}
    Values of Legendre polynomials at $0$ are well-known \cite[Table 18.6.1]{NISTHandbook2010}
    \begin{equation} \label{Eq:values-legendrepoly_zero}
        \begin{dcases} P_{2\ell} (0) = (-1)^{\ell} \frac{\Gamma(\ell + \frac{1}{2})}{\Gamma(\frac{1}{2}) \Gamma(\ell + 1)} \\
        P_{2\ell + 1} (0) = 0 \end{dcases} \ .
    \end{equation}
    Combining \eqref{Eq:identity-Yellgamma-Zellq}, \eqref{Eq:values-alphaell}, and \eqref{Eq:values-legendrepoly_zero}, the result is concluded.
\end{proof}

\begin{proof}[Proof of Proposition \texorpdfstring{\ref{Prop:Existence-Uniqueness-Gzgamma}}{exist-unique-Gzgamma}]
    Recall that the map
    \[
    G^{\gamma} \colon \C \setminus \Speceven(\Lap) \to L^2(\S[2]) \ , \qquad G^{\gamma}(z) = G_{z}^{\gamma}
    \]
    is holomorphic on $\C \setminus \Spec(\Lap)$ because the resolvent $(\Lap {}- z)^{-1}$ is holomorphic on $z \in \C \setminus \Spec(\Lap)$ and $\delta_{\gamma}$ is linear and continuous on $H^{2}(\S[2])$.
    
    Let $z \in \C \setminus \Spec(\Lap)$ and write $G_{z}^{\gamma}$ as a linear combination of eigenfunctions of $\Lap$. Thanks to the defining property \eqref{Eq:def-Gzgamma}, one finds that
    \begin{equation} \label{Eq:EF-expansion-Gzgamma}
        G^{\gamma}(z) = G_{z}^{\gamma} = \sum_{\ell \in \N} \frac{1}{\lambda_{\ell}^2 - z} Y_{\ell}^{\gamma} \ .
    \end{equation}
    We observe now that $Y_{2\ell + 1}^{\gamma} = 0$ for all $\ell \in \N$, thanks to Lemma \ref{Lemma:relation_Yellgamma-Zellq}; therefore, every isolated singularity $z = \lambda_{2 \ell + 1}^{2}$ is a removable singularity, and thus $G^{\gamma}$ extends holomorphically to $\C \setminus \Speceven(\Lap)$ following the formula \eqref{Eq:EF-expansion-Gzgamma}. With formula \eqref{Eq:EF-expansion-Gzgamma} at hand, it is very easy to prove that \eqref{Eq:def-Gzgamma-extension} holds for all $z \in \C \setminus \Speceven(\Lap)$, since
    \[
    \delta_{\gamma} = \sum_{\ell \in \N} Y_{2\ell}^{\gamma} \qquad \text{in $H^{-2}(\S[2])$.} \qedhere
    \]
\end{proof}

The following proposition characterizes the eigenfunctions of the adjoint operator $T^*$, where $T$ was defined back in \eqref{Eq:def-op_T}. Recall the self-adjoint perturbations of $\Lap$ by $\delta_{\gamma}$ are restrictions of the operator $T^*$.

\begin{Prop} \label{Prop:structure-EF-T*}
    Let $\gamma$ be a closed geodesic on $\S[2]$, and let $T \coloneqq \Lap|_{\dom(T)}$ where $\dom(T)$ was defined back in \eqref{Eq:def-op_T}. For every $z \in \C$,
    \begin{equation} \label{Eq:EF-T*}
        \ker(T^* {}- z) = \begin{dcases}
            \ker(\Lap {}- z) + \Span[\C]{G_{z}^{\gamma}}  & \quad \text{if $z \in \C \setminus \Speceven(\Lap)$} \\
            \ker(\Lap {}- z) & \quad \text{if $z \in \Speceven(\Lap)$}
        \end{dcases}
    \end{equation}
\end{Prop}

\begin{proof}
    We will use the spectral decomposition given by $\Lap$: given $u \in H^{s}(\S[2])$, there exists $(u_{\ell})_{\ell \in \N}$ such that $u_{\ell} \in \ker(\Lap {}- \lambda_{\ell}^2)$ $\forall \, \ell \in \N$, 
    \begin{equation} \label{Eq[Prop:structure-EF-T*]:aux1}
        \Big( (1 + \lambda_{\ell}^{2})^{\frac{s}{2}} \norm{u_{\ell}}[L^2(\S[2])] \Big)_{\ell \in \N} \in \ell^2(\N) \qquad \text{and} \qquad u = \sum_{\ell \in \N} u_{\ell} \quad \text{in the $H^s$ topology.}
    \end{equation}
    For example, the spectral decomposition of $\delta_{\gamma}$ is
    \[
    \delta_{\gamma} = \sum_{\ell \in \N} Y_{\ell}^{\gamma} = \sum_{\ell \in \N} Y_{2 \ell}^{\gamma} \ ,
    \]
    but we can also use it to characterize the subspace $\dom(T)$: $v \in \dom(T)$ if and only if its spectral decomposition $(v_{\ell})_{\ell \in \N}$ satisfies \eqref{Eq[Prop:structure-EF-T*]:aux1} for $s = 2$ and also
    \begin{equation} \label{Eq[Prop:structure-EF-T*]:aux2}
        \sum_{\ell \in \N} \ip*{Y_{2\ell}^{\gamma}}{v_{2\ell}}[L^2(\S[2])] = 0 \ .
    \end{equation}

    By definition, $(T^* - z) u_{z} = 0$ if and only if $u_{z} \in L^2(\S[2])$ and $\ip{u_{z}}{(\Lap {}- \conj{z}) v}[L^2(M)] = 0$ for all $v \in \dom(T)$. Translating this into the notation above, there exists $(u_{z, \ell})_{\ell \in \N}$ such that \eqref{Eq[Prop:structure-EF-T*]:aux1} holds for $s = 0$ and
    \[
    \sum_{\ell \in \N} \ip{u_{z, \ell}}{ (\lambda_{\ell}^2 - \conj{z}) v_{\ell} }[L^2(\S[2])] = 0 \qquad \forall \, v \in \dom(T) \ .
    \]
    From this identity we infer that $f \coloneqq \sum_{\ell \in \N} (\lambda_{\ell}^2 - z) u_{z, \ell} \in H^{-2}(\S[2])$ and $f(v) = 0$ for all $v \in \dom(T) = \{v \in H^2(\S[2]) \colon \delta_{\gamma} v = 0\}$. Therefore
    \[
    \sum_{\ell \in \N} (\lambda_{2\ell}^2 - z) u_{z, 2\ell} \in \Span[\C]{ \delta_{\gamma} } \ ,
    \]
    that is, there exists $\beta \in \C$ such that
    \[
    (\lambda_{\ell}^2 - z) u_{z, \ell} = \beta Y_{\ell}^{\gamma} \qquad \forall \, \ell \in \N \ .
    \]

    If $z \in \Speceven(\Lap)$, say $z = \lambda_{2k}^2$, then for $\ell = 2k$ we have the equation
    \[
    0 = (\lambda_{2k}^2 - z) u_{z, 2k} = \beta Y_{2k}^{\gamma} \ ,
    \]
    from which we read that $\beta = 0$, but $u_{z, 2k}$ can be chosen freely. Therefore,
    \[
    u_{z} = w_{z}
    \]
    for some $w_{z} \in \ker(\Lap {}- z)$ arbitrary.

    If $z \in \Specodd(\Lap)$, say $z = \lambda_{2k + 1}^2$, the equation for $\ell = 2 k + 1$
    \[
    0 = (\lambda_{2k+1}^2 - z) u_{z, 2k+1} = \beta Y_{2k+1}^{\gamma} = 0 \ ,
    \]
    imposes no restriction on $\beta$ nor on $u_{z, 2k+1} \in \ker(\Lap {}- \lambda_{2k+1}^2)$. Nonetheless, from the equations for $\ell \neq 2k+1$ we find that
    \[
    u_{z, \ell} = \frac{\beta}{\lambda_{\ell}^2 - z} Y_{\ell}^{\gamma} \ ,
    \]
    therefore
    \[
    u_{z} = w_{z} + \beta G_{z}^{\gamma} \ ,
    \]
    for some $w_z \in \ker(\Lap {}- z)$ arbitrary.

    Lastly, if $z \in \C \setminus \Spec(\Lap)$, then
    \[
    u_{z, \ell} = \frac{\beta}{\lambda_{\ell}^2 - z} Y_{\ell}^{\gamma} \qquad \forall \, \ell \in \N \ ,
    \]
    and thus,
    \[
    u_{z} = \beta G_{z}^{\gamma} \ . \qedhere
    \]
\end{proof}

Next we prove a simple but useful connection between the functions $G_{z}^{\gamma}$ for $z \in \C \setminus \Speceven(\Lap)$ and $u_{\gamma 1}$.

\begin{Lemma} \label{Lemma:Gzgamma-ugamma1-dom(Lapgammatheta)}
    Let $z \in \C \setminus \Speceven(\Lap)$. Then
    \begin{equation} \label{Eq:Gzgamma-ugamma1_in_dom(Lap)}
        G_{z}^{\gamma} - c u_{\gamma 1} \in \dom(\Lap) \quad \iff \quad c = \frac{1}{2\pi} \ .
    \end{equation}
    Therefore, $G_{z}^{\gamma} \in \dom(\Lap_{\gamma, \theta})$ if and only if
    \begin{equation} \label{Eq:char-Gzgamma-in-dom(Lapgammatheta)}
        \delta_{\gamma}\Big( G_{z}^{\gamma} - \frac{1}{2\pi} u_{\gamma 1} \Big) = \frac{1}{2\pi} \cotan \theta \ .
    \end{equation}
\end{Lemma}

\begin{proof}
    We use the following characterization of $\dom(\Lap)$: if $w \in \dom(T^*)$,
    \[
    w \in \dom(\Lap) \quad \iff \quad \omega(u_{\gamma 0}, w) = 0 \ .
    \]
    Let us first prove this. Fix $w \in \dom(T^*)$. Proposition \ref{Prop:computation-dom(T*)} says there exist $w_0 \in \dom(T)$ and $a, b \in \C$ such that
    \[
    w = w_0 + a u_{\gamma 0} + b u_{\gamma 1} \ .
    \]
    Thanks to \eqref{Eq:omega(ugamma01)} and Lemma \ref{Lemma:characterization_dom(T)-via_omega} we see
    \[
    \omega(u_{\gamma 0}, w) = -2\pi b \ .
    \]
    Therefore, $\omega(u_{\gamma 0}, w) = 0$ if and only if $b = 0$, which is equivalent to $w \in \dom(\Lap)$ due to \eqref{Eq:dom(Lap)=dom(T)+Span(ugamma0)}.

    Let us prove now that
    \begin{equation} \label{Eq:aux146}
        \omega \Big( u_{\gamma 0}, G_{z}^{\gamma} - c u_{\gamma 1} \Big) = 0 \quad \iff \quad c = \frac{1}{2\pi} \ .
    \end{equation}
    For that, using the definition of $\omega$ \eqref{Eq:def-omega-sec2}, $u_{\gamma 0} \in \dom(\Lap)$, $T^* G_{z}^{\gamma} = z G_{z}^{\gamma}$ (see Proposition \ref{Prop:structure-EF-T*}), and \eqref{Eq:def-Gzgamma-extension},
    \[
    \omega(u_{\gamma 0}, G_{z}^{\gamma}) = \ip{u_{\gamma 0}}{ z G_{z}^{\gamma}}[L^2(\S[2])] - \ip{\Lap u_{\gamma 0}}{ G_{z}^{\gamma} }[L^2(\S[2])] = \ip{(\conj{z} - \Lap) u_{\gamma 0}}{G_{z}^{\gamma}}[L^2(\S[2])] = - \conj{ \delta_{\gamma} u_{\gamma 0} } = -1 \ .
    \]
    Meanwhile, we know from \eqref{Eq:omega(ugamma01)} that 
    \[
    \omega(u_{\gamma 0}, u_{\gamma 1}) = -2\pi \ .
    \]
    Combining these two identities, we obtain \eqref{Eq:aux146}.

    Lastly, $G_{z}^{\gamma} \in \dom(\Lap_{\gamma, \theta})$ if and only if there exist $g_0 \in \dom(T)$ and $c \in \C$ such that
    \[
    G_{z}^{\gamma} = g_0 + c \big[ \cos \theta \ u_{\gamma 0} + \sin \theta \ u_{\gamma 1} \big] \ .
    \]
    We rewrite it as follows
    \[
    G_{z}^{\gamma} - c \sin \theta \ u_{\gamma 1} = g_0 + c \cos \theta \ u_{\gamma 0} \in \dom(\Lap) \ .
    \]
    Thanks to \eqref{Eq:Gzgamma-ugamma1_in_dom(Lap)}, we get that $G_{z}^{\gamma} \in \dom(\Lap_{\gamma, \theta})$ if and only if (observe that $\sin \theta \neq 0$)
    \[
    G_{z}^{\gamma} - \frac{1}{2\pi} u_{\gamma 1} = g_0 + \frac{1}{2\pi} \cotan \theta \ u_{\gamma 0} \ ,
    \]
    or as a single equation,
    \[
    \delta_{\gamma} \Big( G_{z}^{\gamma} - \frac{1}{2\pi} u_{\gamma 1} \Big) = \frac{1}{2\pi} \cotan \theta \ . \qedhere
    \]
\end{proof}

\begin{proof}[Proof of Theorem \texorpdfstring{\ref{Thm:Spectrum-Lapgammatheta}}{compact resolvent}]
    Let $z \in \C \setminus (\Spec(\Lap_{\gamma, \theta}) \cup \Spec(\Lap))$ and fix some $f \in L^2(\S[2])$. Since $\Lap$ and $\Lap_{\gamma, \theta}$ are self-adjoint, there exist unique $u \in H^2(\S[2])$ and $w \in \dom(T^*)$ such that
    \[
    (\Lap {}- z) u = f \ , \qquad (\Lap_{\gamma, \theta} {}- z) [u + w] = f \ .
    \]
    This directly implies that
    \begin{equation} \label{Eq:resolvent-difference-f}
        \Big[ (\Lap_{\gamma, \theta} {}- z)^{-1} - (\Lap {}- z)^{-1} \Big] f = w \ .
    \end{equation}
    Since both $\Lap$ and $\Lap_{\gamma, \theta}$ are restrictions of the operator $T^*$, we may apply $(T^* - z)$ on both sides of \eqref{Eq:resolvent-difference-f} and obtain $w \in \ker(T^* - z)$. Now Proposition \ref{Prop:structure-EF-T*} implies that
    \[
    w = \beta G_{z}^{\gamma}
    \]
    for some $\beta \in \C$ depending on $f$. This proves that
    \[
    (\Lap_{\gamma, \theta} {}- z)^{-1} - (\Lap {}- z)^{-1} \in \Span[\C]{G_{z}^{\gamma}}
    \]
    that is, it is a rank-one operator, thus $\Lap_{\gamma, \theta}$ has compact resolvent and $z \in \Spec(\Lap_{\gamma, \theta})$ if and only if $z$ is an eigenvalue of $\Lap_{\gamma, \theta}$.

    To compute $\beta$, we use that $u + w \in \dom(\Lap_{\gamma, \theta})$, which means that there exist $u_0 \in \dom(T)$ and $c \in \C$ such that
    \begin{equation} \label{Eq:u+bGzgamma=u_0thetaugmma01}
        u + \beta G_{z}^{\gamma} = u_0 + c \big[ \cos \theta \ u_{\gamma 0} + \sin \theta \ u_{\gamma 1} \big]
    \end{equation}
    where $u_{\gamma 0}$, $u_{\gamma 1}$ were defined in \eqref{Eq:functions-ugamma01}. Let us rewrite \eqref{Eq:u+bGzgamma=u_0thetaugmma01} as follows,
    \[
    \beta G_{z}^{\gamma} - c \sin \theta \ u_{\gamma 1} = u_0 + c \cos \theta \ u_{\gamma 0} - u \in \dom(\Lap) \ .
    \]
    Due to Lemma \ref{Eq:Gzgamma-ugamma1_in_dom(Lap)}, we get that $c = \frac{\beta}{2 \pi \sin \theta}$ (recall that we assumed $\theta \neq 0$). Therefore, applying $\delta_{\gamma}$ on both sides of \eqref{Eq:u+bGzgamma=u_0thetaugmma01},
    \[
    \delta_{\gamma} u + \beta \delta_{\gamma} \Big( G_{z}^{\gamma} - \frac{1}{2\pi} u_{\gamma 1} \Big) = \frac{\beta}{2\pi} \cotan \theta \ ,
    \]
    and solving for $\beta$,
    \begin{equation} \label{Eq:aux6345}
        \beta = \frac{1}{\cotan \theta - \delta_{\gamma} ( G_{z}^{\gamma} - \frac{1}{2\pi} u_{\gamma 1} )} \delta_{\gamma} u \ ,
    \end{equation}
    where we used that $\cotan \theta - \delta_{\gamma} ( G_{z}^{\gamma} - \frac{1}{2\pi} u_{\gamma 1} )$ because $z \notin \Spec(\Lap_{\gamma, \theta})$ (see Lemma \ref{Lemma:Gzgamma-ugamma1-dom(Lapgammatheta)}). Lastly, using that
    \[
    \delta_{\gamma} u = \ip{G_{\conj{z}}^{\gamma}}{(\Lap {}- z) u}[L^2(\S[2])] = \ip{G_{\conj{z}}^{\gamma}}{f}[L^2(\S[2])] \ ,
    \]
    and combining it with \eqref{Eq:aux6345}, we obtain
    \[
    \beta = \frac{1}{\cotan \theta - \delta_{\gamma} ( G_{z}^{\gamma} - \frac{1}{2\pi} u_{\gamma 1} )} \ip{G_{\conj{z}}^{\gamma}}{f}[L^2(\S[2])] \ ,
    \]
    thus proving identity \eqref{Eq:resolvent-identity-Lapgammatheta}.

    The rest of the identities in Theorem \ref{Thm:Spectrum-Lapgammatheta} are consequences of the more general identity
    \begin{equation} \label{Eq:generalformula-Spec(Lapgammatheta)}
        \ker(\Lap_{\gamma, \theta} {}- \eta) = \ker(T^* - \eta) \cap \dom(\Lap_{\gamma, \theta}) \qquad \forall \, \eta \in \R \ ,
    \end{equation}
    in combination with Proposition \ref{Prop:structure-EF-T*}. Let us start considering $\eta \in \R \setminus \Speceven(\Lap)$, Proposition \ref{Prop:structure-EF-T*} gives
    \[
    \ker(T^* - \eta) = \ker(\Lap - \eta) + \Span[\C]{G_{\eta}^{\gamma}} \ .
    \]
    We have $\delta_{\gamma} w_{\eta} = 0$ for all $w_{\eta} \in \ker(\Lap {}- \eta)$, thus $\ker(\Lap {}- \eta) \subseteq \dom(\Lap_{\gamma, \theta})$. If $\eta$ is not a solution to \eqref{Eq:equation_eta-newEV}, then $G_{\eta}^{\gamma} \notin \dom(\Lap_{\gamma, \theta})$ due to \eqref{Eq:char-Gzgamma-in-dom(Lapgammatheta)} and thus
    \[
    \ker(\Lap_{\gamma, \theta} {}- \eta) = \ker(\Lap {}- \eta)
    \]
    If $\eta$ is a solution to \eqref{Eq:equation_eta-newEV}, then $G_{\eta}^{\gamma} \in \dom(\Lap_{\gamma, \theta})$, therefore
    \[
    \ker(\Lap_{\delta, \gamma} {}- \eta) = \ker(\Lap {}- \eta) + \Span[\C]{G_{z}^{\gamma}} \ .
    \]
    If we consider now $\eta \in \Speceven(\Lap)$, say $\eta = \lambda_{2\ell}^2$, Proposition \ref{Prop:structure-EF-T*} implies that
    \[
    \ker(T^* {}- \eta) = \ker(\Lap - \eta) \ .
    \]
    Since $Y_{2\ell}^{\gamma} \in \ker(\Lap {}- \eta) \setminus \ker(\Lap_{\gamma, \theta} {}- \eta)$, we immediately obtain
    \[
    \ker(\Lap_{\gamma, \theta} {}- \lambda_{2\ell}^2) = \big\{ u \in \ker(\Lap {}- \lambda_{2\ell}^2) \colon \delta_{\gamma} u = 0 \big\} \ . \qedhere
    \]
\end{proof}

\begin{Lemma} \label{Lemma:EV-equation}
    Fix $m \in \N$, and consider the interval $J_{m} \coloneqq (\lambda_{2m - 2}^2, \lambda_{2m}^2)$ if $m \geq 1$, and $J_{0} \coloneqq (- \infty, 0)$ if $m = 0$. The function
    \[
    F \colon J_{m} \to \R \ , \qquad F(\eta) \coloneqq \delta_{\gamma} \Big( G_{\eta}^{\gamma} - \frac{1}{2\pi} u_{\gamma 1} \Big)
    \]
    is strictly increasing, and if $m \geq 1$, it is bijective.
\end{Lemma}

\begin{proof}
    From the eigenfunction expansion of $G_{\eta}^{\gamma}$ for $\eta \in \R$ \eqref{Eq:EF-expansion-Gzgamma}, we see that $G_{\eta}^{\gamma}$ is real-valued (because $Y_{\ell}^{\gamma}$ is real-valued for all $\ell \in \N$), and since $u_{\gamma 1}$ is real-valued too \eqref{Eq:functions-ugamma01}, then $F(\eta) \in \R$ for all $\eta \in J_{m}$. Moreover, Proposition \ref{Prop:Existence-Uniqueness-Gzgamma} gives that $G^{\gamma}$ is analytic on $\C \setminus \Speceven(\Lap)$, therefore $F$ is analytic on $J_{m}$, and thus differentiable. We compute $F'(\eta)$, for $\eta \in J_{m}$.
    \begin{align*}
        F'(\eta) & = \delta_{\gamma} \bigg( \frac{\D{}}{\D{\tau}} \left. \Big( G_{\tau}^{\gamma} - u_{\gamma 1} \Big) \right|_{\tau = \eta} \bigg) = \delta_{\gamma} \bigg( \sum_{\ell \in \N} \frac{1}{(\lambda_{2\ell}^{2} - \eta)^2} Y_{2 \ell}^{\gamma} \bigg) \\
        & = \sum_{\ell \in \N} \frac{1}{(\lambda_{2\ell}^{2} - \eta)^2} \delta_{\gamma}Y_{2 \ell}^{\gamma} = \sum_{\ell \in \N} \frac{1}{(\lambda_{2\ell}^{2} - \eta)^2} \norm*{Y_{2 \ell}^{\gamma}}[L^2(\S[2])]^2 = \norm*{G_{\eta}^{\gamma}}[L^2(\S[2])]^2 > 0 \ .
    \end{align*}
    This proves that $F$ is strictly increasing on $J_{m}$. To show that it is bijective, it suffices to show that $\lim_{\eta \to \lambda_{2m}^2} F(\eta) = +\infty$ and $\lim_{\eta \to \lambda_{2m - 2}^{2}} F(\eta) = -\infty$ for $m \geq 1$.

    Let us define the function
    \[
    H^{\gamma}(z) \coloneqq G_{z}^{\gamma} - u_{\gamma 1} - \frac{1}{\lambda_{2m}^{2} - z} Y_{2m}^{\gamma} \ , \qquad z \in \C \setminus \Speceven(\Lap) \ .
    \]
    Since $\Lap(G_{z}^{\gamma} - u_{\gamma 1}) = zG_{z}^{\gamma} - 2u_{\gamma 1}$, we see that $H^{\gamma}$ is a $H^2(\S[2])$-valued analytic function with a removable singularity at $z = \lambda_{2m}^{2}$. Therefore
    \begin{align*}
        \lim_{\eta \to \lambda_{2m}^{2}} F(\eta) & = \lim_{\eta \to \lambda_{2m}^{2}} \delta_{\gamma} \bigg( H^{\gamma}(\eta) + \frac{1}{\lambda_{2m}^{2} - \eta} Y_{2m}^{\gamma} \bigg) \\
        & = \delta_{\gamma} \Big( H^{\gamma}(\lambda_{2m}^{2}) \Big) + \norm*{Y_{2m}^{\gamma}}[L^2(\S[2])]^2 \lim_{\eta \to \lambda_{2m}^{2}} \frac{1}{\lambda_{2m}^{2} - \eta} = + \infty \ .
    \end{align*}
    The statement $\lim_{\eta \to \lambda_{2\ell - 2}^{2}} F(\eta) = -\infty$ is proven similarly.
\end{proof}

\section{Quasimodes for new eigenfunctions}
\label{Sec:QuasimodesforGF}

In this section, we are going to find quasimodes of $\Lap$ that approximate $G_{\eta}^{\gamma}$ in the high-energy limit. This is the semiclassical limit; thus, we instead introduce a semiclassical parameter $h > 0$ and write everything in terms of it.

Given $h > 0$ such that $h^{-2} \notin \Speceven(\Lap)$, we will denote it throughout the rest of the paper
\begin{equation} \label{Eq:def-Ghgamma}
    G_{h}^{\gamma} \coloneqq \sum_{\ell \in \N} \frac{1}{\lambda_{2\ell}^2 - h^{-2}} Y_{2\ell}^{\gamma} \ ,
\end{equation}
and
\begin{equation} \label{Eq:def-ghgamma}
    g_{h}^{\gamma} \coloneqq \frac{1}{\norm*{G_{h}^{\gamma}}[L^2(\S[2])]} G_{h}^{\gamma} \ .
\end{equation}
We look for quasimodes of $g_{h}^{\gamma}$ of the form
\[
\frac{1}{\norm*{G_{h}^{\gamma}}[L^2(\S[2])]} \sum_{\ell \in I_h} \frac{1}{\lambda_{2\ell}^2 - h^{-2}} Y_{2\ell}^{\gamma} \ ,
\]
for some finite interval $I_h$ depending on $h$. For that, we need to understand the asymptotics of
\[
\frac{1}{\norm*{G_{h}^{\gamma}}[L^2(\S[2])]} \qquad \text{as $h \to 0^+$,}
\]
and of the $L^2$-norm of the terms
\begin{equation*} 
    \frac{1}{\lambda_{2\ell}^2 - h^{-2}} Y_{2\ell}^{\gamma} \qquad \text{as $h \to 0^+$ and $\ell$ depends on $h$.}
\end{equation*}
Therefore, given a sequence $(h_{n})_{n \in \N} \subseteq (0, \infty)$ such that $h_n^{-2} \notin \Speceven(\Lap)$, for every $n \in \N$ we introduce $\ell_n \in \N$ and $\sigma_n \in (0,1)$ defined by the following identity
\begin{equation} \label{Eq:def-elln+sigman-gamma}
    (h_n)^{-2} = \big( 2(\ell_n + \sigma_n) \big) \big( 2(\ell_n + \sigma_n) + 1 \big) \ .
\end{equation}
One has $(h_n)^{-2} \in (\lambda_{2\ell_n}^2 , \lambda_{2\ell_n + 2}^2)$, and the parameter $\sigma_n$ measures the proximity of $(h_n)^{-2}$ to the ends of this interval, which will help us control the coefficients $( \lambda_{2\ell}^2 - h^{-2} )^{-1}$. Lastly, for $\ell \in \N$, define
\begin{equation} \label{Eq:normalized-Yellgamma}
    y_{2\ell}^{\gamma} \coloneqq \frac{1}{\norm{Y_{2\ell}^{\gamma}}[L^2(\S[2])]} Y_{2\ell}^{\gamma} \ ,
\end{equation}

\begin{Theorem} \label{Thm:quasimodes-Ghgamma}
    Let $(h_n)_{n \in \N} \subseteq (0,1)$ such that $h_n^{-2} \notin \Speceven(\Lap)$, and let $\ell_n \in \N$ and $\sigma_n \in (0,1)$ be defined as in \eqref{Eq:def-elln+sigman-gamma}. Assume that $\lim_{n \to \infty} h_n = 0$ and that $\lim_{n \to \infty} \sigma_n = \sigma \in [0,1]$.
    \begin{enumerate}
        \item If $\sigma \in (0,1)$, set
        \begin{equation} \label{Eq:Csigma-definition}
             C_{\sigma} \coloneqq \bigg( \sum_{k \in \Z} \frac{1}{(k - \sigma)^2} \bigg)^{\frac{1}{2}} \ .
         \end{equation}
        There exists $C > 1$ such that for every $\Upsilon \geq 1$ there exists $N \in \N$ such that for all $n \geq N$,
        \[
        \norm*{ g_{h_n}^{\gamma} - \frac{1}{C_{\sigma}} \sum_{\abs{k - \sigma} \leq \Upsilon} \frac{1}{k - \sigma} y_{2(\ell_n + k)}^{\gamma} }[L^2(\S[2])] \leq C \Big[\Upsilon h_n + \abs{\sigma_n - \sigma} + \Upsilon^{-\frac{1}{2}} \Big] \ .
        \]
        \item If $\sigma \in \{0, 1\}$, there exists $C > 1$ and $N \in \N$ such that for all $n \geq N$,
        \[
        \norm*{ g_{h_n}^{\gamma} - (-1)^{1 - \sigma} y_{2(\ell_n + \sigma)}^{\gamma}}[L^2(\S[2])] \leq C \Big[ h_n + \abs{\sigma_n - \sigma} \Big] \ . 
        \]
    \end{enumerate}
\end{Theorem}

Theorem \ref{Thm:quasimodes-Ghgamma} will be q consequence of the asymptotics provided in Propositions \ref{Prop:Ghgamma_tails-L2norm-estimate} and \ref{Prop:Ghgamma_mainterm-L2norm}, which themselves depend on several simpler asymptotics, which we collect in lemmas below.

\begin{Lemma} \label{Lemma:norm-Yellgamma-asymptotic}
    Let $\gamma$ be a closed geodesic in $\S[2]$, and let $Y_{\ell}^{\gamma}$ be the eigenfunction defined in \eqref{Eq:def-Yellgamma}. There exists $C > 0$ such that for all $\ell \in \N$
    \begin{equation} \label{Eq:norm_asymptotic-Y2lgamma}
        \abs*{ \norm*{Y_{2\ell}^{\gamma}}[L^2(\S[2])]^2 - \frac{1}{\pi^2} } \leq C (1 + \ell)^{-1} \ .
    \end{equation}
\end{Lemma}

\begin{proof}
    Lemma \ref{Lemma:relation_Yellgamma-Zellq} implies that for all $\ell \in \N$
    \[
    Y_{2\ell}^{\gamma} = (-1)^{\ell} \frac{\Gamma(\frac{1}{2} + \ell)}{\Gamma(\frac{1}{2}) \Gamma(\ell + 1)} Z_{2\ell}^{q} \ .
    \]
    Moreover, recall that
    \[
    \norm*{Z_{\ell}^{q}}[L^2(\S[2])]^2 = Z_{\ell}^{q}(q) = \frac{m_{\ell}}{\vol(\S[2])} = \frac{2\ell + 1}{\vol(\S[2])} \ .
    \]
    Putting both things together, we find that
    \begin{equation*}
        \norm*{Y_{\ell}^{\gamma}}[L^2(\S[2])]^2 = \frac{1}{\Gamma(\frac{1}{2})^2 \vol(\S[2])} \frac{\Gamma(\ell + \frac{1}{2})^2 (4\ell + 1)}{\Gamma(\ell + 1)^2} \ .
    \end{equation*}
    Lastly, we may invoke the following asymptotic estimate for the Gamma function \cite[(5.11.3)]{NISTHandbook2010},
    \[
    \frac{\Gamma(x + a)}{\Gamma(x + b)} = x^{a-b}(1 + \BigO(x^{-1})) \qquad \text{as $x \to +\infty$,}
    \]
    and the identities $\Gamma(1/2) = \sqrt{\pi}$ and $\vol(\S[2]) = 4\pi$, to obtain
    \[
    \abs*{\norm*{Y_{\ell}^{\gamma}}[L^2(\S[2])]^2 - \frac{1}{\pi^2} } = \BigO(\ell^{-1}) \qquad \text{as $\ell \to \infty$.}
    \]
    Since $\N$ is discrete in $\R$, one can find $C > 0$ such that \eqref{Eq:norm_asymptotic-Y2lgamma} holds for all $\ell \in \N$
\end{proof}

\begin{Lemma} \label{Lemma:hn-elln-approximation-gamma}
    For $h > 0$ such that $h^{-2} \notin \Spec(\Lap)$, set $\ell \in \N$ and $\sigma \in (0,1)$ such that
    \[
    h^{-2} = \big( 2(\ell + \sigma) \big) \big( 2(\ell + \sigma) + 1\big) \in (\lambda_{2\ell}^2, \lambda_{2\ell + 2}^2) \ .
    \]
    Then, for every $0 < h < 1$
    \begin{equation*}
        \abs*{h^{-1} - [2(\ell + \sigma) + \tfrac{1}{2}] } \leq h \ , \qquad \abs*{h - \frac{1}{2(\ell + \sigma) + \frac{1}{2} } } \leq h^3 \ .
    \end{equation*}
\end{Lemma}

\begin{proof}
    The first inequality follows from a direct computation:
    \[
    \abs*{h^{-1} - (2(\ell + \sigma) + \tfrac{1}{2}) } = \frac{\abs*{h^{-2} - [2(\ell + \sigma) + \frac{1}{2}]^2 } }{ h^{-1} + [2(\ell + \sigma) + \frac{1}{2}] } \leq \frac{ \frac{1}{4} }{ h^{-1} } \leq h \ .
    \]
    The second inequality is a consequence of the first. We see that
    \[
    2(\ell + \sigma) + \tfrac{1}{2} \geq h^{-1} \ ,
    \]
    therefore,
    \[
    \abs*{h - \frac{1}{2(\ell + \sigma) + \frac{1}{2}} } = \frac{ \abs*{h^{-1} - [2(\ell + \sigma) + \tfrac{1}{2}] } }{ h^{-1} [2(\ell + \sigma) + \tfrac{1}{2}] } \leq \frac{h}{h^{-2}} = h^3 \ . \qedhere
    \]
\end{proof}

\begin{Lemma} \label{Lemma:lambdaell+k-asymptotic-gamma}
    For $h > 0$ such that $h^{-2} \notin \Spec(\Lap)$, set $\ell \in \N$ and $\sigma \in (0,1)$ such that
    \[
    h^{-2} = \big( 2(\ell + \sigma) \big) \big( 2(\ell + \sigma) + 1\big) \in (\lambda_{2\ell}^2, \lambda_{2\ell + 2}^2) \ .
    \]
    Then, for every $0 < h < 1$ and every $k \in \Z$, $k \geq -\ell$,
    \begin{equation} \label{Eq:lambdaell+k-asymptotic}
        \abs*{ \lambda_{2(\ell + k)}^2 - \Big[ h^{-2} + 4(k - \sigma) h^{-1} + 4(k - \sigma)^2 \Big]} \leq 4 \abs{k - \sigma} h \ .
    \end{equation}
\end{Lemma}

\begin{proof}
    Using the identity $x(x + 1) - y(y + 1) = (x - y)(x + y + 1)$, we find that
    \begin{align*}
        \lambda_{2(\ell + k)}^2 - h^{-2} & = \big( 2(\ell + k) \big) \big( 2(\ell + k) + 1\big) - \big( 2(\ell + \sigma) \big) \big( 2(\ell + \sigma) + 1\big) \\
        & = 2(k - \sigma) [ 4\ell + 2k + 2\sigma + 1] \ .
    \end{align*}
    Writing
    \[
    4\ell + 2k + 2\sigma + 1 = 2\Big[ 2(\ell + \sigma) + \tfrac{1}{2} \Big] + 2(k - \sigma) \ ,
    \]
    one finds that
    \[
    \lambda_{2(\ell + k)}^2 - h^{-2} - 4(k - \sigma)h^{-1} - 4(k - \sigma)^2 = 4(k - \sigma)\Big[ 2(\ell + \sigma) + \tfrac{1}{2} - h^{-1} \Big] \ .
    \]
    Now Lemma \ref{Lemma:hn-elln-approximation-gamma} implies
    \[
    \abs*{ \lambda_{2(\ell + k)}^2 - h^{-2} - 4(k - \sigma) h^{-1} - 4(k - \sigma)^2 } \leq 4\abs{k - \sigma} h \ . \qedhere
    \]
\end{proof}

\begin{Lemma} \label{Lemma:asymptotic-summand-gamma}
    For $h > 0$ such that $h^{-2} \notin \Spec(\Lap)$, set $\ell \in \N$ and $\sigma \in (0,1)$ such that
    \[
    h^{-2} = \big( 2(\ell + \sigma) \big) \big( 2(\ell + \sigma) + 1\big) \in (\lambda_{2\ell}^2, \lambda_{2\ell + 2}^2) \ .
    \]
    There exists $C > 0$ and $h_0 > 0$  such that for all $0 < h < h_0$ and all $k \geq -\ell$
    \begin{multline*}
        \abs*{ \frac{1}{\abs*{\lambda_{2(\ell + k)}^2 - h^{-2}}^2 } \norm*{Y_{2(\ell + k)}^{\gamma}}[L^2(\S[2])]^2 - \frac{1}{16(k - \sigma)^2} \frac{1}{ (h^{-1} + k - \sigma)^2} \frac{1}{\pi^2} } \\
        \leq \frac{C}{\abs{k - \sigma}^2} \Big[ h^2 \abs{1 + \ell + k}^{-1} + h^{4} + h^{5} \abs{k - \sigma} \Big] \ .
    \end{multline*}
\end{Lemma}

\begin{proof}
    This is a consequence of the triangular inequality, after adding and subtracting
    \[
    \frac{1}{\big[ 4 (k - \sigma) \, ( h^{-1} + k - \sigma) \big]^2 } \norm*{Y_{2(\ell + k)}^{\gamma}}[L^2(\S[2])]^2 \ .
    \]
    On the one hand, thanks to Lemma \ref{Lemma:norm-Yellgamma-asymptotic},
    \begin{multline*}
        \abs*{ \frac{1}{\abs*{\lambda_{2(\ell + k)}^2 - h^{-2}}^2 } \norm*{Y_{2(\ell + k)}^{\gamma}}[L^2(\S[2])]^2 - \frac{1}{\big[ 4 (k - \sigma) \, (h^{-1} + k - \sigma) \big]^2 } \norm*{Y_{2(\ell + k)}^{\gamma}}[L^2(\S[2])]^2 } \\
        \leq C \abs*{ \frac{ \big[ \lambda_{2 (\ell + k)}^2 - h^{-2} \big]^2 - \big[ 4 (k - \sigma) \, (h^{-1} + k - \sigma) \big]^2 }{ \big[ (\lambda_{2 (\ell + k)}^2 - h^{-2}) 4 (k - \sigma) \, (h^{-1} + k - \sigma) \big]^2 } } \ .
    \end{multline*}
    For the numerator, we may apply identity $x^2 - y^2 = (x - y) (x + y)$ and Lemma \ref{Lemma:lambdaell+k-asymptotic-gamma} to obtain
    \begin{multline*}
    \abs*{ \big[ \lambda_{2 (\ell + k)}^2 - h^{-2} \big]^2 - \big[ 4 (k - \sigma) \, (h^{-1} + k - \sigma) \big]^2 } \\
    \leq 4 \abs{k - \sigma} h \ \abs*{ \lambda_{2 (\ell + k)}^2 - h^{-2} + 4(k - \sigma) h^{-1} + 4(k - \sigma)^2} \\
    \leq C \abs{k - \sigma} h \Big[ \abs*{k - \sigma} h^{-1} + \abs{k - \sigma}^2 \Big] \ .
    \end{multline*}
    For the denominator, thanks to Lemma \ref{Lemma:lambdaell+k-asymptotic-gamma},
    \begin{multline*}
        \abs*{ \lambda_{2 (\ell + k)}^2 - h^{-2} } \geq \abs*{ 4(k - \sigma) (h^{-1} + k - \sigma) } - \abs*{\lambda_{2(\ell + k)}^2 - \Big[ h^{-2} + 4(k - \sigma) h^{-1} + 4(k - \sigma)^2 \Big] } \\
        \geq 4 \abs{k - \sigma} (h^{-1} + k - \sigma) - 4 \abs{k - \sigma} h = 4\abs{k - \sigma} (h^{-1} + k - \sigma - h) \ .
    \end{multline*}
    thus, using the fact that there exists $h_0 > 0$ such that 
    \begin{equation} \label{Eq:aux-inequality}
        h^{-1} + k - \sigma - h \geq \frac{1}{2} h^{-1} \qquad \forall \ 0 < h < h_0 \ ,
    \end{equation}
    one finds that
    \[
    \abs*{(\lambda_{2 (\ell + k)}^2 - h^{-2}) 4 (k - \sigma) \, (h^{-1} + k - \sigma) }^2 \geq 16 \abs{k - \sigma}^4 h^{-4} \ .
    \]
    Putting everything together,
    \begin{multline*}
        \abs*{ \frac{1}{\abs*{\lambda_{2(\ell + k)}^2 - h^{-2}}^2 } \norm*{Y_{2(\ell + k)}^{\gamma}}[L^2(\S[2])]^2 - \frac{1}{\big[ 4 (k - \sigma) \, (h^{-1} + k - \sigma) \big]^2 } \norm*{Y_{2(\ell + k)}^{\gamma}}[L^2(\S[2])]^2 } \\
        \leq \frac{ C \abs{k - \sigma} h }{ 16 \abs{k - \sigma}^4 h^{-4} } \Big[ \abs*{k - \sigma} h^{-1} + \abs{k - \sigma}^2 \Big] \leq C \frac{1}{\abs{k - \sigma}^2} \Big[h^4 + h^5 \abs{k - \sigma} \Big]
    \end{multline*}
    
    On the other hand, Lemma \ref{Lemma:norm-Yellgamma-asymptotic} and inequality \eqref{Eq:aux-inequality} let us obtain
    \begin{multline*}
        \abs*{ \frac{1}{\big[ 4 (k - \sigma) \, ( h^{-1} + k - \sigma) \big]^2 } \norm*{Y_{2(\ell + k)}^{\gamma}}[L^2(\S[2])]^2 - \frac{1}{16 (k - \sigma)^2} \frac{1}{(h^{-1} + k - \sigma)^2} \frac{1}{\pi^2} } \\
        \leq \frac{C}{4 (k - \sigma)^2} h^2 \abs{1 + \ell + k}^{-1} \ .
    \end{multline*}
    The triangular inequality gives the result.
\end{proof}

For a set $A \subseteq \N$, let $\Pi_{A}$ be the spectral eigenprojector onto the subspace $\oplus_{\ell \in A} \ker(\Lap {}- \lambda_{\ell}^2)$, and for $s \in \R$, $r \geq 0$, let $I(s, r) \coloneqq [s - r, s + r] \cap \N$. For example, for $h_n > 0$ and $\ell_n \in \N$ as defined above, and $\Upsilon \in \N$,
\begin{equation} \label{Eq:example-cutoff_Upsilon-Ghgamma}
    \Pi_{I(2\ell_n, 2\Upsilon)} G_{h_n}^{\gamma} = \sum_{\abs{k} \leq \Upsilon} \frac{1}{ \lambda_{2(\ell_n + k)}^2 - h_n^{-2} } Y_{2(\ell_n + k)}^{\gamma} \ .
\end{equation}

\begin{Prop} \label{Prop:Ghgamma_tails-L2norm-estimate}
    Let $(h_n)_{n \in \N} \subseteq (0,1)$ be such that $h_n^{-2} \notin \Speceven(\Lap)$, and for every $n \in \N$, let $\ell_n \in \N$ and $\sigma_n \in (0,1)$ be defined by \eqref{Eq:def-elln+sigman-gamma}. Assume that $\lim_{n \to \infty} h_n = 0$. There exists $C > 1$ and $h_0 > 0$ such that for all integers $\Upsilon \geq 1$, and all $0 < h_n < h_0$
    \begin{equation*} 
        \norm*{G_{h_n}^{\gamma} - \Pi_{I(2(\ell_n + \sigma_n), 2\Upsilon)} G_{h_n}^{\gamma}}[L^2(\S[2])]^{2} \leq C \frac{(h_n)^{2}}{\Upsilon} \ .
    \end{equation*}
\end{Prop}

\begin{proof}
    Note that since $\sigma_n \in (0,1)$ for all $n \in \N$, then
    \[
    I(2(\ell_n + \sigma_n), 2\Upsilon) \cap 2\N = \{ 2(\ell_n + k) \in \N \colon -(\Upsilon - 1) \leq k \leq \Upsilon \} \ .
    \]
    Therefore, we have to estimate
    \[
    \norm*{G_{h_n}^{\gamma} - \Pi_{I(2(\ell_n + \sigma_n), 2\Upsilon)} G_{h_n}^{\gamma}}[L^2(\S[2])]^{2} = \bigg[ \sum_{k = -\ell_n}^{-\Upsilon} + \sum_{k = \Upsilon + 1}^{\infty} \bigg] \frac{1}{ \abs*{\lambda_{2(\ell_n + k)}^2 - h_n^{-2}}^2 } \norm*{Y_{2(\ell_n + k)}^{\gamma}}[L^2(\S[2])]^2
    \]
    as $n \to \infty$. For that, we invoke Lemma \ref{Lemma:asymptotic-summand-gamma} for $-\ell_n \leq k \leq -\Upsilon$ and $k \geq \Upsilon + 1$. First, we see that for all $0 < h_n < h_0$
    \begin{align*}
        \hspace{30pt} \abs*{ \frac{1}{16(k - \sigma_n)^2} \frac{1}{ (h_n^{-1} + k - \sigma_n)^2} \frac{1}{\pi^2} } & \leq C \frac{h_n^2}{k^2} \ , && \text{if $-\ell_n \leq k \leq - \Upsilon$,} \hspace{30pt}\\
        \abs*{ \frac{1}{16(k - \sigma_n)^2} \frac{1}{ (h_n^{-1} + k - \sigma_n)^2} \frac{1}{\pi^2} } & \leq C \frac{h_n^2}{(k - 1)^2} \ , && \text{if $k \geq \Upsilon + 1$.}
    \end{align*}
    Consequently, Lemma \ref{Lemma:asymptotic-summand-gamma} implies that for all $0 < h_n < h_0$,
    \begin{align*}
        \hspace{20pt} \frac{1}{\abs*{\lambda_{2(\ell_n + k)}^2 - h_{n}^{-2}}^2 } \norm*{Y_{2(\ell_n + k)}^{\gamma}}[L^2(\S[2])]^2 & \leq C \frac{1}{\abs{k}^2} (h_n)^2 \ , && \text{if $-\ell_n \leq k \leq -\Upsilon$,} \hspace{20pt} \\
        \frac{1}{\abs*{\lambda_{2(\ell_n + k)}^2 - h_{n}^{-2}}^2 } \norm*{Y_{2(\ell_n + k)}^{\gamma}}[L^2(\S[2])]^2 & \leq C \frac{1}{\abs{k - 1}^2} (h_n)^2 \ , && \text{if $k \geq \Upsilon + 1$.}
    \end{align*}
    Therefore,
    \[
    \norm*{G_{h_n}^{\gamma} - \Pi_{I(2(\ell_n + \sigma_n), 2\Upsilon)} G_{h_n}^{\gamma}}[L^2(\S[2])]^{2} \leq C \bigg( \sum_{k \geq \Upsilon} \frac{1}{k^2} \bigg) (h_n)^2 \leq C \frac{(h_n)^2}{\Upsilon} \ . \qedhere
    \]
\end{proof}

\begin{Prop} \label{Prop:Ghgamma_mainterm-L2norm}
    Let $(h_n)_{n \in \N} \subseteq (0,1)$ be such that $h_n^{-2} \notin \Speceven(\Lap)$, and for every $n \in \N$, let $\ell_n \in \N$ and $\sigma_n \in (0,1)$ be defined by \eqref{Eq:def-elln+sigman-gamma}. Assume that $\lim_{n \to \infty} h_n = 0$ and that $\lim_{n \to \infty} \sigma_n = \sigma \in [0,1]$.
     \begin{enumerate}
         \item If $\sigma \in (0,1)$, set
         \begin{equation} \label{Eq:Csigma-definition-rep}
             C_{\sigma} \coloneqq \bigg( \sum_{k \in \Z} \frac{1}{(k - \sigma)^2} \bigg)^{\frac{1}{2}} \ .
         \end{equation}
         There exists $C > 1$ such that for every $\Upsilon \geq 1$ there exists $N \in \N$ such that for all $n \geq N$,
         \begin{equation} \label{Eq:Ghgamma_mainterm-L2norm-sigma(01)}
             \abs*{ \norm*{\Pi_{I(2(\ell_n + \sigma), 2\Upsilon)} G_{h_n}^{\gamma} }[L^2(\S[2])]^2 - \frac{(C_{\sigma})^2}{16\pi^2} (h_n)^2  } \leq C (h_n)^2 \Big[\Upsilon h_n + \abs{\sigma_n - \sigma} + \Upsilon^{-1} \Big] \ .
         \end{equation}
         \item If $\sigma \in \{0, 1\}$, there exists $C > 1$ and $N \in \N$ such that for all $n \geq N$,
         \begin{equation} \label{Eq:Ghgamma_mainterm-L2norm-sigma0.1}
             \abs*{ \norm*{\Pi_{I(2(\ell_n + \sigma), 0)} G_{h_n}^{\gamma}}[L^2(\S[2])]^2 - \frac{1}{16\pi^2} \frac{(h_n)^2}{\abs{\sigma_n - \sigma}^2} } \leq C \frac{(h_n)^3}{\abs{\sigma_n - \sigma}^2 } \ .
         \end{equation}
     \end{enumerate}
\end{Prop}

\begin{proof}
    \textbf{(1)} First assume that $\sigma \in (0,1)$. We have
    \[
    \norm*{\Pi_{I(2(\ell_n + \sigma), 2\Upsilon)} G_{h_n}^{\gamma} }[L^2(\S[2])]^2 = \sum_{k = -\Upsilon + 1}^{\Upsilon} \frac{1}{\abs*{\lambda_{2(\ell_n + k)}^2 - h_{n}^{-2}}^2 } \norm*{Y_{2(\ell_n + k)}^{\gamma}}[L^2(\S[2])]^2
    \]
    therefore,
    \begin{align*}
        & \abs*{ \norm*{\Pi_{I(2(\ell_n + \sigma), 2\Upsilon)} G_{h_n}^{\gamma} }[L^2(\S[2])]^2 - \frac{(C_{\sigma})^2}{16\pi^2} (h_n)^2 } \\
        & \hspace{90pt} \leq \sum_{k = -\Upsilon + 1}^{\Upsilon} \abs[\Bigg]{ \frac{1}{\abs*{\lambda_{2(\ell_n + k)}^2 - h_{n}^{-2}}^2 } \norm*{Y_{2(\ell_n + k)}^{\gamma}}[L^2(\S[2])]^2 - \frac{(h_n)^2}{16 (k - \sigma)^2} \frac{1}{\pi^2} } \\
        & \hspace{120pt} + \frac{(h_n)^2}{16 \pi^2} \bigg[ \sum_{k = - \ell_n}^{-\Upsilon} + \sum_{k = \Upsilon + 1}^{\infty} \bigg] \frac{1}{\abs{k - \sigma}^2} \\
        & \hspace{90pt} \eqqcolon \sum_{k = - \Upsilon + 1}^{\Upsilon} A_k + B_{\Upsilon} \ .
    \end{align*}
    It is immediate to prove that for all $\Upsilon \geq 1$,
    \[
    B_{\Upsilon} \leq C \frac{(h_n)^2}{\Upsilon} \ .
    \]

    To estimate $A_{k}$ for $-\Upsilon + 1 \leq k \leq \Upsilon$, we combine triangular inequality with Lemma \ref{Lemma:asymptotic-summand-gamma}: there exists $C > 0$ and $h_0 > 0$ such that for all $0 < h_n < h_0$ and $-\Upsilon + 1 \leq k \leq \Upsilon$,
    \[
    A_{k} \leq \frac{C}{\abs{k - \sigma_n}^2} \Big[ (h_n)^3 + \Upsilon (h_n)^{5} \Big] + \frac{1}{16 \pi^2} \abs*{ \frac{1}{\abs{k - \sigma_n}^2} \frac{1}{\abs{h_n^{-1} + k - \sigma_n}^2} - \frac{(h_n)^2}{\abs{k - \sigma}^2} } \ .
    \]
    We estimate the second term invoking triangular inequality once again, adding and subtracting
    \[
    \frac{1}{(k - \sigma)^2} \frac{1}{\abs{h_n^{-1} + k - \sigma_n}^2} \ . 
    \]
    Take $N \in \N$ such that $h_n < h_0$ and $\abs{\sigma_n - \sigma} \leq \frac{1}{2} d(\sigma, \Z)$ for all $n \geq N$. One finds
    \begin{equation} \label{Eq:aux-inequlity-sigma-sigman}
        \begin{aligned}
            \abs*{ \frac{1}{(k - \sigma_n)^2} - \frac{1}{(k - \sigma)^2} } & = \frac{1}{(k - \sigma)^2} \abs*{ \frac{1}{(1 + \frac{\sigma - \sigma_n}{k - \sigma})^2 } - 1 } = \frac{1}{(k - \sigma)^2} \frac{\abs*{\frac{\sigma - \sigma_n}{k - \sigma}} \abs*{2 + \frac{\sigma - \sigma_n}{k - \sigma} } }{ (1 + \frac{\sigma - \sigma_n}{k - \sigma})^2 } \\
            & \leq 10 \frac{\abs{\sigma - \sigma_n}}{\abs{k - \sigma}^3} \ ,
        \end{aligned}
    \end{equation}
    thus, triangular inequality and \eqref{Eq:aux-inequality} gives
    \[
    \abs*{ \frac{1}{\abs{k - \sigma_n}^2} \frac{1}{\abs{h_n^{-1} + k - \sigma_n}^2} - \frac{(h_n)^2}{\abs{k - \sigma}^2} } \leq 10 \frac{\abs{\sigma - \sigma_n}}{\abs{k - \sigma}^3} (h_n)^2 + \frac{1}{(k - \sigma)^2} \abs*{ \frac{1}{\abs{h_n^{-1} + k - \sigma_n}^2} - (h_n)^{2} } \ .
    \]
    Lastly, since $-\Upsilon + 1 \leq k \leq \Upsilon$, one finds
    \[
    \abs*{ \frac{1}{\abs{h_n^{-1} + k - \sigma_n}^2} - (h_n)^{2} } \leq C \Upsilon (h_n)^3 \ .
    \]
    Putting everything together, we have
    \[
    A_{k} \leq C (h_n)^2 \Big[ \frac{\abs{\sigma - \sigma_n}}{\abs{k - \sigma}^3} + \frac{h_n \Upsilon}{\abs{k - \sigma}^2} \Big] \ ,
    \]
    where we applied \eqref{Eq:aux-inequlity-sigma-sigman} to change $\abs{k - \sigma_n}$ for $\abs{k - \sigma}$ in the denominator when needed. Combining the estimates for $A_{k}$ and $B_{\Upsilon}$, and using that
    \[
    \sum_{k \in \Z} \frac{1}{\abs{k - \sigma}^2} < \infty \qquad \text{and} \qquad \sum_{k \in \Z} \frac{1}{\abs{k - \sigma}^3} < \infty \ ,
    \]
    we obtain \eqref{Eq:Ghgamma_mainterm-L2norm-sigma(01)}.

    \textbf{(2)} Let us assume now that $\sigma \in \{0,1\}$. We have
    \[
    \norm*{\Pi_{I(2(\ell_n + \sigma), 0)} G_{h_n}^{\gamma}}[L^2(\S[2])]^2 = \frac{1}{\abs*{ \lambda_{2(\ell_n + \sigma)}^2 - (h_n)^{-2} }^2 } \norm*{Y_{2(\ell_n + \sigma)}^{\gamma}}[L^2(\S[2])]^2 \ .
    \]
    A direct application of Lemma \ref{Lemma:asymptotic-summand-gamma} for $k = \sigma$ and triangular inequality gives
    \begin{multline*}
        \abs*{ \norm*{\Pi_{I(2(\ell_n + \sigma), 0)} G_{h_n}^{\gamma}}[L^2(\S[2])]^2 - \frac{1}{16\pi^2} \frac{(h_n)^2}{\abs{\sigma_n - \sigma}^2} } \\
        \leq \frac{C}{\abs{\sigma - \sigma_n}^2} (h_n)^3 + \frac{1}{16\pi^2 \abs{\sigma - \sigma_n}^2 } \abs*{ \frac{1}{\abs{h_n^{-1} + \sigma - \sigma_n}^2} - (h_n)^2 } \ .
    \end{multline*}
    Since $\sigma_n, \sigma \in [0,1]$, we have
    \[
    \abs*{ \frac{1}{\abs{h_n^{-1} + \sigma - \sigma_n}^2} - (h_n)^2 } \leq C (h_n)^3 \ ,
    \]
    hence, we have shown that
    \[
    \abs*{ \norm*{\Pi_{I(2(\ell_n + \sigma), 0)} G_{h_n}^{\gamma}}[L^2(\S[2])]^2 - \frac{1}{16\pi^2} \frac{(h_n)^2}{\abs{\sigma_n - \sigma}^2} } \leq C \frac{(h_n)^3}{\abs{\sigma - \sigma_n}^2} \ . \qedhere
    \]
\end{proof}

\begin{proof}[Proof of Theorem \texorpdfstring{\ref{Thm:quasimodes-Ghgamma}}{quasimodes ghgamma}]
    \textbf{(1)} Assume that $\sigma \in (0,1)$ and let $\Upsilon \geq 1$. Note that $I(2(\ell_n + \sigma_n), 2\Upsilon) = I(2(\ell_n + \sigma), 2\Upsilon)$. As a first approximation for $g_{h_n}^{\gamma}$ we have
    \begin{equation} \label{Eq:ghngamma-cutoffapprox}
        \norm*{g_{h_n}^{\gamma} - \Pi_{I(2(\ell_n + \sigma), 2\Upsilon)} g_{h_n}^{\gamma} }[L^2(\S[2])] \leq C \Upsilon^{-\frac{1}{2}} \ ,
    \end{equation}
    thanks to Propositions \ref{Prop:Ghgamma_tails-L2norm-estimate} and \ref{Prop:Ghgamma_mainterm-L2norm}. We use Pythagorean theorem to estimate the remaining term in the approximation,
    \begin{multline*}
        \norm*{\Pi_{I(2(\ell_n + \sigma), 2\Upsilon)} g_{h_n}^{\gamma} - \frac{1}{C_{\sigma}} \sum_{\abs{k - \sigma} \leq \Upsilon} \frac{1}{k - \sigma} y_{2(\ell_n + k)}^{\gamma} }[L^2(\S[2])] \\
        = \Bigg( \sum_{\abs{k - \sigma} \leq \Upsilon} \abs[\Bigg]{ \frac{\norm[\big]{Y_{2(\ell_n + k)}^{\gamma}}[L^2(\S[2])] }{\norm[\big]{G_{h_n}^{\gamma}}[L^2(\S[2])] (\lambda_{2(\ell_n + k)}^2 - h_n^{-2})} - \frac{1}{C_{\sigma}} \frac{1}{k - \sigma} }^2 \Bigg)^{\frac{1}{2}} \ .
    \end{multline*}
    In order to upper bound
    \[
    \abs[\Bigg]{ \frac{\norm[\big]{Y_{2(\ell_n + k)}^{\gamma}}[L^2(\S[2])] }{\norm[\big]{G_{h_n}^{\gamma}}[L^2(\S[2])] (\lambda_{2(\ell_n + k)}^2 - h_n^{-2})} - \frac{1}{C_{\sigma}} \frac{1}{k - \sigma} }
    \]
    for $\abs{k - \sigma} \leq \Upsilon$ we are going to make repeated use of triangular inequality, making partial approximations on each step, until we get the simplified term $\frac{1}{C_{\sigma}} \frac{1}{k - \sigma}$. First, thanks to Lemma \ref{Lemma:asymptotic-summand-gamma} and Proposition \ref{Prop:Ghgamma_mainterm-L2norm} we have
    \[
    \abs[\Bigg]{ \frac{\norm[\big]{Y_{2(\ell_n + k)}^{\gamma}}[L^2(\S[2])] }{\norm[\big]{G_{h_n}^{\gamma}}[L^2(\S[2])] (\lambda_{2(\ell_n + k)}^2 - h_n^{-2})} - \frac{1}{\norm{G_{h_n}^{\gamma}}[L^2(\S[2])]} \frac{1}{4\pi (k - \sigma_n)} \frac{1}{\abs{h_n^{-1} + k - \sigma_n}} } \leq \frac{C}{\abs{k - \sigma_n}} h_n \ .
    \]
    Since $\sigma_n \to \sigma$ and $\abs{k - \sigma} \leq \Upsilon$, and thanks to Proposition \ref{Prop:Ghgamma_mainterm-L2norm}, we obtain
    \[
    \abs*{\frac{1}{\norm{G_{h_n}^{\gamma}}[L^2(\S[2])]} \frac{1}{4\pi (k - \sigma_n)} \frac{1}{\abs{h_n^{-1} + k - \sigma_n}} - \frac{1}{\norm{G_{h_n}^{\gamma}}[L^2(\S[2])]} \frac{ h_n }{4\pi (k - \sigma)} } \leq \frac{C}{\abs{k - \sigma}} \Big[ \abs{\sigma_n - \sigma} + \Upsilon h_n \Big] \ .
    \]
    Lastly, thanks to Proposition \ref{Prop:Ghgamma_mainterm-L2norm} we get
    \[
    \abs*{ \frac{1}{\norm{G_{h_n}^{\gamma}}[L^2(\S[2])]} \frac{ h_n }{4\pi (k - \sigma)} - \frac{1}{C_{\sigma}} \frac{1}{k - \sigma} } \leq \frac{C}{\abs{k - \sigma} } \Big[ \Upsilon h_n + \abs{\sigma_n - \sigma} + \Upsilon^{-1} \Big] \ .
    \]
    Putting everything together,
    \[
    \abs[\Bigg]{ \frac{\norm[\big]{Y_{2(\ell_n + k)}^{\gamma}}[L^2(\S[2])] }{\norm[\big]{G_{h_n}^{\gamma}}[L^2(\S[2])] (\lambda_{2(\ell_n + k)}^2 - h_n^{-2})} - \frac{1}{C_{\sigma}} \frac{1}{k - \sigma} } \leq \frac{C}{\abs{k - \sigma} } \Big[ \Upsilon h_n + \abs{\sigma_n - \sigma} + \Upsilon^{-1} \Big] \ ,
    \]
    hence, using that $\sum_{k \in \Z} \frac{1}{\abs{k - \sigma}^2} < \infty$,
    \begin{equation} \label{Eq:ghgamma-finalapprox}
        \norm*{\Pi_{I(2(\ell_n + \sigma), 2\Upsilon)} g_{h_n}^{\gamma} - \frac{1}{C_{\sigma}} \sum_{\abs{k - \sigma} \leq \Upsilon} \frac{1}{k - \sigma} y_{2(\ell_n + k)}^{\gamma} }[L^2(\S[2])] \leq C \Big[ \Upsilon h_n + \abs{\sigma_n - \sigma} + \Upsilon^{-1} \Big] \ .
    \end{equation}
    Combining \eqref{Eq:ghngamma-cutoffapprox} and \eqref{Eq:ghgamma-finalapprox}, we have proved the desired estimate in the case $\sigma \in (0,1)$.

    \textbf{(2)} Assume that $\sigma \in \{0,1\}$. Thanks to Proposition \ref{Prop:Ghgamma_tails-L2norm-estimate} with $\Upsilon = 1$ and Lemma \ref{Lemma:asymptotic-summand-gamma} for $k = 1 - \sigma$, in combination with Proposition \ref{Prop:Ghgamma_mainterm-L2norm}, we obtain for large enough $n \in \N$
    \begin{equation} \label{Eq:aux3602}
        \norm*{g_{h_n}^{\gamma} - \Pi_{I(2(\ell_n + \sigma), 0)} g_{h_n}^{\gamma} }[L^2(\S[2])] \leq C \abs{\sigma - \sigma_n} \ .
    \end{equation}
    In addition, using that $(-1)^{1 - \sigma} = \frac{\abs{\sigma - \sigma_n}}{\sigma - \sigma_n}$ for all $n \in \N$,
    \[
    \norm*{ \Pi_{I(2(\ell_n + \sigma), 0)} g_{h_n}^{\gamma} - (-1)^{1 - \sigma} y_{2(\ell_n + \sigma)}^{\gamma} }[L^2(\S[2])] = \abs*{ \frac{\norm[\big]{Y_{2(\ell_n + \sigma)}^{\gamma}}[L^2(\S[2])] }{\norm[\big]{G_{h_n}^{\gamma}}[L^2(\S[2])] (\lambda_{2(\ell_n + \sigma)}^2 - h_n^{-2})} - \frac{\abs{\sigma - \sigma_n}}{\sigma - \sigma_n} } \ .
    \]
    Adding and subtracting
    \[
    \frac{1}{\norm[\big]{G_{h_n}^{\gamma}}[L^2(\S[2])]} \frac{h_n}{4\pi (\sigma - \sigma_n)}
    \]
    and applying triangular inequality, we find that
    \begin{align*}
    \abs*{ \frac{\norm[\big]{Y_{2(\ell_n + \sigma)}^{\gamma}}[L^2(\S[2])] }{\norm[\big]{G_{h_n}^{\gamma}}[L^2(\S[2])] (\lambda_{2(\ell_n + \sigma)}^2 - h_n^{-2})} - \frac{\abs{\sigma - \sigma_n}}{\sigma - \sigma_n} } & \leq \frac{1}{\norm[\big]{G_{h_n}^{\gamma}}[L^2(\S[2])]} \frac{C (h_n)^2}{\abs{\sigma - \sigma_n}} \\
    & \qquad + \frac{h_n}{4\pi \abs{\sigma - \sigma_n}} \abs*{ \frac{1}{\norm[\big]{G_{h_n}^{\gamma}}[L^2(\S[2])]} - \frac{4\pi \abs{\sigma - \sigma_n}}{h_n}}
    \end{align*}
    thanks to Lemma \ref{Lemma:asymptotic-summand-gamma} for $k = \sigma$. Now, \eqref{Eq:Ghgamma_mainterm-L2norm-sigma0.1} gives
    \begin{equation} \label{Eq:aux4629}
        \abs*{ \frac{\norm[\big]{Y_{2(\ell_n + \sigma)}^{\gamma}}[L^2(\S[2])] }{\norm[\big]{G_{h_n}^{\gamma}}[L^2(\S[2])] (\lambda_{2(\ell_n + \sigma)}^2 - h_n^{-2})} - \frac{\abs{\sigma - \sigma_n}}{\sigma - \sigma_n} } \leq C \Big[ h_n + \abs{\sigma - \sigma_n} \Big] \ .
    \end{equation}
    Combining estimates \eqref{Eq:aux3602} and \eqref{Eq:aux4629} we obtain the desired estimate in the case $\sigma \in \{0, 1\}$.
\end{proof}

\section{Proofs of main theorems}
\label{Sec:Proofs_maintheorems}

In this section, we state and prove the results related to the semiclassical measures of high-energy eigenfunctions of singular perturbations of $\Lap$ by $\delta_{\gamma}$. In the introduction, these were denoted by
\[
\Lap + \alpha \ket{\delta_{\gamma}}\bra{\delta_{\gamma}} \ , \qquad \alpha \in \R \ ,
\]
but we saw in Proposition \ref{Prop:family-perturbation-Lap-gamma} that these are parameterized instead by $\theta \in \R / (\pi \Z)$, and we denote them by $\Lap_{\gamma, \theta}$. Theorem \ref{Thm:maintheorem-invariantSDM} is a consequence of Theorem \ref{Thm:equal-set_SDM}, and Theorem \ref{Thm:maintheorem-QuantumLimit} is a consequence of Theorem \ref{Thm:end-theorem-non-invariance} and Lemma \ref{Lemma:projection_into_QL}.

Recall the definition of $Y_{\ell}^{\gamma}$ \eqref{Eq:def-Yellgamma}, $Z_{\ell}^{q}$ \eqref{Eq:def-Zellq}, and $G_{h}^{\gamma}$ \eqref{Eq:def-Ghgamma}. In this section, we will work mainly with their normalized versions; those are
\begin{equation} \label{Eq:def-normalized-notable-functions}
    y_{2\ell}^{\gamma} \coloneqq \frac{1}{\norm{Y_{2\ell}^{\gamma}}[L^2(\S[2])]} Y_{2\ell}^{\gamma} \ , \qquad z_{\ell}^{q} \coloneqq \frac{1}{\norm{Z_{\ell}^{q}}[L^2(\S[2])]} Z_{\ell}^{q} \ , \qquad g_{h}^{\gamma} \coloneqq \frac{1}{\norm*{G_{h}^{\gamma}}[L^2(\S[2])]} G_{h}^{\gamma} \ .
\end{equation}
Note that if $q$ and $\gamma$ are such that $d(q, \gamma) = \frac{\pi}{2}$, then Proposition \ref{Lemma:relation_Yellgamma-Zellq} implies that
\begin{equation} \label{Eq:y2ellgamma=(-1)elll_z2ellq}
    y_{2\ell}^{\gamma} = (-1)^{\ell} z_{2\ell}^{q} \qquad \forall \, \ell \in \N \ .
\end{equation}

For $q \in \S[2]$, let us define the probability measure $\mu_{q}$ on $T^*\S[2]$,
\begin{equation} \label{Eq:def-nuq}
     \int_{T^*\S[2]} a(x, \xi) \mu_{q}(\D{x}, \D{\xi}) \coloneqq \int_{S_{q}^*\S[2]} \int_{0}^{2\pi} a(\phi_t(q, \xi)) \frac{\D{t}}{2\pi} \frac{\D{\xi}}{2\pi} \ , \qquad \forall \, a \in \CinfK(T^*\S[2]) \ .
\end{equation}
Note that $\mu_{q} \in \mathcal{P}_{\mathrm{inv}}(S^*\S[2])$, that is, it is supported inside $S^*\S[2]$ and is invariant under geodesic flow; in fact, $\mu_q$ is the unique semiclassical measure associated to the sequence of normalized zonal harmonics on $q \in \S[2]$ (see Theorem 3.1 of \cite{Verdasco2026spheres}).

\begin{Theorem} \label{Thm:SDM-newEF_is_muq}
    Let $\gamma$ be a closed geodesic on $\S[2]$, and let $q \in \S[2]$ be such that $d(q, \gamma) = \frac{\pi}{2}$. Let $(h_n)_{n \in \N} \subseteq (0,1)$ be such that $\lim_{n \to \infty} h_n = 0$ and $(h_n)^{-2} \in \R \setminus \Speceven(\Lap)$. For every $a \in \CinfK(T^*\S[2])$,
    \[
    \lim_{n \to \infty} \ip{ g_{h_n}^{\gamma}}{ \Op[h_n]{a} g_{h_n}^{\gamma} }[L^2(\S[2])] = \int_{T^*\S[2]} a \, \mu_{q}(\D{x}, \D{\xi}) \ .
    \]
\end{Theorem}

\begin{proof}
    We prove the sequence of Wigner distributions
    \[
    W_{h_n}[g_{h_n}^{\gamma}] (a) \coloneqq \ip{ g_{h_n}^{\gamma} }{\Op[h_n]{a} g_{h_n}^{\gamma} }[L^2(\S[2])] \ ,
    \]
    converges to the measure $\mu_q$ in the weak-$\star$ topology of $\Dist(T^*\S[2])$ by showing that every subsequence of $(W_{h_n}[ g_{h_n}^{\gamma} ])_{n \in \N}$ has a further subsequence that converges to $\mu_q$.
    
    Given any subsequence of Wigner distributions as described above, there exists a further subsequence, which we do not relabel, and a Radon measure $\mu$ on $T^*\S[2]$ such that (see \cite[Theorem 5.2]{Zworski2012})
    \begin{equation} \label{Eq:existance_SDM_mu}
        \lim_{n \to \infty} \ip{ g_{h_n}^{\gamma} }{\Op[h_n]{a} g_{h_n}^{\gamma} }[L^2(\S[2])] = \int_{T^*\S[2]} a \, \mu \qquad \forall \, a \in \CinfK(T^*\S[2]) \ .
    \end{equation}
    Up to extraction of a further subsequence, we may assume that $\sigma_n \to \sigma \in [0,1]$. We will show that $\mu = \mu_{q}$ under this additional convergence of the sequence $(\sigma_n)_{n \in \N}$.

    \textbf{A. Case $\sigma \in \{0, 1\}$.} Theorem \ref{Thm:quasimodes-Ghgamma}(2) says that there exists $C > 0$ and $N \in \N$ such that for all $n \geq N$,
    \[
    \norm*{ g_{h_n}^{\gamma} - (-1)^{1 - \sigma} y_{2(\ell_n + \sigma)}^{\gamma}}[L^2(\S[2])] \leq C \Big[ h_n + \abs{\sigma_n - \sigma} \Big] \ . 
    \]
    If we choose $q \in \S[2]$ such that $d(q, \gamma) = \frac{\pi}{2}$, then \eqref{Eq:y2ellgamma=(-1)elll_z2ellq} implies that for all $n \geq N$
    \[
    \norm*{ g_{h_n}^{\gamma} - (-1)^{1 + \ell_n} z_{2(\ell_n + \sigma)}^{q}}[L^2(\S[2])] \leq C \Big[ h_n + \abs{\sigma_n - \sigma} \Big] \ . 
    \]
    Consequently, we have the following approximation as $n \to \infty$:
    \[
    \ip{ g_{h_n}^{\gamma} }{\Op[h_n]{a} g_{h_n}^{\gamma} }[L^2(\S[2])] = \ip{ z_{2 (\ell_n + \sigma)}^{q} }{\Op[h_n]{a} z_{2(\ell_n + \sigma)}^{q} }[L^2(\S[2])] + \BigO(h_n) + \BigO(\abs{\sigma_n - \sigma}) \ .
    \]
    We now take limits $n \to \infty$ and invoke Theorem 3.1 from \cite{Verdasco2026spheres} to obtain
    \[
    \int_{T^*\S[2]} a \, \mu = \lim_{n \to \infty} \ip{ g_{h_n}^{\gamma} }{\Op[h_n]{a} g_{h_n}^{\gamma} }[L^2(\S[2])] = \int_{T^*\S[2]} a \, \mu_{q} \ ,
    \]
    as we wanted to show.

    \textbf{B. Case $\sigma \in (0, 1)$.} Let $\Upsilon \in \N$ large; we will take $\Upsilon \to \infty$ at the end. Thanks to Theorem \ref{Thm:quasimodes-Ghgamma}, there exist $C_{\sigma} > 0$ \eqref{Eq:Csigma-definition}, $C > 0$, and $N \in \N$ such that for all $n \geq N$,
    \[
    \norm*{ g_{h_n}^{\gamma} - \frac{1}{C_{\sigma}} \sum_{\abs{k - \sigma} \leq \Upsilon} \frac{1}{k - \sigma} y_{2(\ell_n + k)}^{\gamma} }[L^2(\S[2])] \leq C \Big[\Upsilon h_n + \abs{\sigma_n - \sigma} + \Upsilon^{-\frac{1}{2}} \Big] \ .
    \]
    Let us fix $q \in \S[2]$ such that $d(q, \gamma) = \frac{\pi}{2}$, thus \eqref{Eq:y2ellgamma=(-1)elll_z2ellq} implies that for all $n \geq N$,
    \begin{equation} \label{Eq:aprox-ghgamma-finlincombzellq}
        \norm*{ g_{h_n}^{\gamma} - (-1)^{\ell_n} \frac{1}{C_{\sigma}} \sum_{\abs{k - \sigma} \leq \Upsilon} \frac{ (-1)^{k} }{k - \sigma} z_{2(\ell_n + k)}^{q} }[L^2(\S[2])] \leq C \Big[\Upsilon h_n + \abs{\sigma_n - \sigma} + \Upsilon^{-\frac{1}{2}} \Big] \ .
    \end{equation}
    For $r = d(x, q)$, let
    \begin{equation} \label{Eq:def-kappaq-varsigmaq}
        \kappa^{q}(x, \xi) \coloneqq \cos r \ , \qquad \varsigma^{q}(x, \xi) \coloneqq \xi (\sin r \cdot \partial_{r}) \ , \qquad (x, \xi) \in T^*\S[2] \ ,
    \end{equation}
    and define the smooth symbol on $T^*\S[2]$,
    \begin{equation} \label{Eq:def-GammasigmaUpsilon}
        \Gamma^{\sigma, \Upsilon} \coloneqq \frac{1}{C_{\sigma}} \bigg[ \sum_{k = -\Upsilon + 1}^{0} \frac{(-1)^{k}}{k - \sigma} (\kappa^{q} - i \varsigma^{q})^{-(2k - 1)} + \sum_{k = 1}^{\Upsilon} \frac{(-1)^{k}}{ k - \sigma } (\kappa^{q} + i \varsigma^{q})^{2k - 1} \bigg] \ .
    \end{equation}
    Theorem 5.1 from \cite{Verdasco2026spheres} for $x_n = X \in \ell_2(\Z)$, 
    \[
    X = (X_{j})_{j \in \Z} \coloneqq \begin{dcases}
        \frac{1}{C_{\sigma}} \frac{ (-1)^k }{k - \sigma} & \quad j = 2k \\
        0 & \quad j = 2k + 1 
    \end{dcases} \ , \qquad  S = S(\Upsilon) \coloneqq [-2\Upsilon + 1, 2 \Upsilon - 1] \cap \N \ ,
    \]
    and $\rho_n \coloneqq 2\ell_n + 1$ implies that there exist $C(\Upsilon) > 0$ and $N = N(\Upsilon) \in \N$ such that for all $n \geq N$
    \begin{equation} \label{Eq:finlincombzellq-approx-PsDO(symbol)zelln}
        \norm*{ \frac{1}{C_{\sigma}} \sum_{\abs{k - \sigma} \leq \Upsilon} \frac{ (-1)^{k} }{k - \sigma} z_{2(\ell_n + k)}^{q} - \Op[h_n]{\Gamma^{\sigma, \Upsilon} } z_{2\ell_n + 1}^{q} }[L^2(\S[2])] \leq C \Upsilon^{-\frac{1}{2}} + C(\Upsilon) h_n \ .
    \end{equation}
    Combining \eqref{Eq:aprox-ghgamma-finlincombzellq} and \eqref{Eq:finlincombzellq-approx-PsDO(symbol)zelln} we obtain that for all $n \geq N$,
    \begin{equation} \label{Eq:ghgamma-approx-PsDO(symbol)zelln}
        \norm*{ g_{h_n}^{\gamma} - (-1)^{\ell_n} \Op[h_n]{\Gamma^{\sigma, \Upsilon} } z_{2\ell_n + 1}^{q} }[L^2(\S[2])] \leq C \Big[ C(\Upsilon) h_n + \abs{\sigma_n - \sigma} + \Upsilon^{-\frac{1}{2}} \Big] \ .
    \end{equation}
    Consequently, for every $\Upsilon \in \N$ and every $n \geq N(\Upsilon)$
    \begin{multline*} 
        \abs*{ \ip*{ g_{h_n}^{\gamma} }{\Op[h_n]{a} g_{h_n}^{\gamma} }[L^2(\S[2])] - \ip*{ \Op[h_n]{\Gamma^{\sigma, \Upsilon}} z_{2 \ell_n + 1}^{q} }{\Op[h_n]{a } \Op[h_n]{\Gamma^{\sigma, \Upsilon}} z_{2\ell_n + 1}^{q} }[L^2(\S[2])] } \\
        \leq C \Big[ C(\Upsilon) h_n + \abs{\sigma_n - \sigma} + \Upsilon^{-\frac{1}{2}} \Big] \ .
    \end{multline*}
    Using symbolic calculus and taking $n \to \infty$ thanks to Theorem 3.1 from \cite{Verdasco2026spheres}, we find that for every $\Upsilon \in \N$,
    \begin{equation} \label{Eq:approx-preDCT}
        \abs*{ \int_{T^*\S[2]} a \, \mu - \int_{T^*\S[2]} a \cdot \abs*{\Gamma^{\sigma, \Upsilon} }^2 \mu_{q} } \leq C \Upsilon^{-\frac{1}{2}} \ .
    \end{equation}
    We would like to take $\Upsilon \to \infty$ in this last identity, but for that we first need to study the weak-$\star$ limits of sequence of measures $(\abs*{\Gamma^{\sigma, \Upsilon} }^2 \mu_{q})_{\Upsilon \in \N}$.

    For every $a \in \CinfK(T^*\S[2])$ we have
    \[
    \int_{T^*\S[2]} a \cdot \abs*{\Gamma^{\sigma, \Upsilon} }^2 \mu_{q} = \int_{S_{q}^*\S[2]} \int_{0}^{2\pi} a(\phi_t(q, \xi)) \abs*{\Gamma^{\sigma, \Upsilon} (\phi_t(q, \xi)) }^2 \frac{\D{t}}{2\pi} \frac{\D{\xi}}{ \vol(S_{q}^{*}\S[2]) } \ ,
    \]
    where $\phi_t$ denote the geodesic flow at time $t \in \R$. From the definition \eqref{Eq:def-kappaq-varsigmaq} we see that
    \begin{equation} \label{Eq:kappaq-varsigmaq-flowoutq}
        \kappa^{q}(\phi_t(q, \xi)) = \cos t \ , \qquad \varsigma^{q}(\phi_{t}(q, \xi)) = \sin t \qquad \forall \, \xi \in S_{q}^*\S[2] \ , \quad t \in \R \ ,
    \end{equation}
    therefore,
    \[
    \Gamma^{\sigma, \Upsilon}(\phi_t(q, \xi)) = \frac{e^{-it}}{C_{\sigma}} \sum_{k = -\Upsilon + 1}^{\Upsilon} \frac{(-1)^k}{ k - \sigma } e^{i 2kt} \ .
    \]

    \begin{Lemma} \label{Lemma:zetasigma2-Fourierexp}
        For $\sigma \in (0,1)$, let $\zeta^{\sigma, 1} \in L^2(\R / (2\pi \Z))$ be defined by
        \[
        \zeta^{\sigma, 1} (t) \coloneqq \frac{1}{C_{\sigma}} \frac{2\pi i}{1 - e^{i 2 \pi \sigma} } e^{i \sigma t} \ , \qquad t \in [0, 2\pi) \ ,
        \]
        which is discontinuous at $t = 0 \in \R / (2\pi \Z)$. Then the function $\zeta^{\sigma, 2} \in L^2(\R / (2\pi \Z))$ given by
        \begin{equation} \label{Eq:def-zetasigma2}
            \zeta^{\sigma, 2} (t) \coloneqq \begin{dcases}
                \zeta^{\sigma, 1}(2t + \pi) & \qquad t \in [0, \tfrac{\pi}{2}) \\
                \zeta^{\sigma, 1}(2t - \pi) & \qquad t \in [\tfrac{\pi}{2}, \tfrac{3\pi}{2})\\
                \zeta^{\sigma, 1}(2t - 3\pi) & \qquad t \in [\tfrac{3\pi}{2}, 2\pi)
            \end{dcases}
        \end{equation}
        has Fourier series
        \[
        \zeta^{\sigma, 2}(t) = \frac{1}{C_{\sigma}} \sum_{k \in \Z} \frac{(-1)^{k}}{k - \sigma } e^{i 2k t} \ .
        \]
    \end{Lemma}
    \begin{proof}[Proof of Lemma \texorpdfstring{\ref{Lemma:zetasigma2-Fourierexp}}{zetasigma2}]
        One first proves that $\zeta^{\sigma, 1}$ has Fourier series
        \[
        \zeta^{\sigma, 1} (t) = \frac{1}{C_{\sigma}} \sum_{k \in \Z} \frac{1}{k - \sigma} e^{i k t} \ .
        \]
        One obtains as a consequence that
        \begin{equation} \label{Eq:computation-Csigma}
            C_{\sigma} = \frac{2\pi}{\abs{1 - e^{i 2\pi \sigma}}} \ .
        \end{equation}
        Now, the Fourier series of $\zeta^{\sigma, 2}$ can be computed from the Fourier series of $\zeta^{\sigma, 1}$ because $\zeta^{\sigma, 2} = D_{2} S_{-\pi} \zeta^{\sigma, 1}$, where $D_{2}$ and $S_{-\pi}$ are the unitary maps on $L^2(\R /(2\pi \Z))$
        \[
        [D_{2} f] (t) \coloneqq f(2t) \ , \qquad [S_{-\pi} f](t) \coloneqq f(t + \pi) \ . \qedhere
        \]
    \end{proof}

    From Lemma \ref{Lemma:zetasigma2-Fourierexp} we infer that $\Gamma^{\sigma, \Upsilon}(\phi_t(q, \xi))$ is the partial Fourier series of $\zeta^{\sigma, 2}$ \eqref{Eq:def-zetasigma2}. Thanks to Carleson's theorem we get that
    \[
    \Gamma^{\sigma, \Upsilon}(\phi_t(q, \xi)) \xrightarrow[\Upsilon \to \infty]{} e^{-it} \zeta^{\sigma, 2}(t) \quad \text{a.e. $t \in \R / (2\pi \Z)$,}
    \]
    therefore, due to identity \eqref{Eq:computation-Csigma},
    \[
    \abs*{ \Gamma^{\sigma, \Upsilon}(\phi_t(q, \xi)) }^2 \xrightarrow[\Upsilon \to \infty]{} \abs*{ \zeta^{\sigma, 2}(t) }^2 = 1 \quad \text{a.e. $t \in \R / (2\pi \Z)$.}
    \]
    Now we take $\Upsilon \to \infty$ in \eqref{Eq:approx-preDCT} and apply the dominated convergence theorem to get
    \[
    \int_{T^*\S[2]} a \, \mu = \int_{T^*\S[2]} a \, \mu_{q} \ . \qedhere 
    \]
\end{proof}

Let $\gamma$ be a closed geodesic on $\S[2]$, and for a given $\theta \in \Theta = \R / (\pi \Z)$, let $\Lap_{\gamma, \theta}$ be the singular perturbation of $\Lap$ by $\delta_{\gamma}$ as defined in \eqref{Eq:dom(Lapgammatheta)}. Consider the equation for $\eta \in \R \setminus \Speceven$,
\begin{equation} \label{Eq:equation_eta-newEV-rep}
    \delta_{\gamma} \Big( G_{\eta}^{\gamma} - \frac{1}{2\pi} u_{\gamma 1} \Big) = \frac{1}{2\pi} \cotan \theta \ .
\end{equation}
We define the set of solutions $\eta$ to \eqref{Eq:equation_eta-newEV-rep}
\begin{equation} \label{Eq:set-newEV}
    \Lambda_{\theta} \coloneqq \big\{ \eta \in \R \setminus \Speceven(\Lap) \colon \text{$\eta$ satisfies \eqref{Eq:equation_eta-newEV-rep}} \, \big\} \ .
\end{equation}
The set $\Lambda_{\theta}$ is countable as a consequence of Lemma \ref{Lemma:EV-equation}.

Given a sequence of singular perturbations of $\Lap$ by $\delta_{\gamma}$, $(\Lap_{\gamma, \theta_{n}})_{n \in \N}$, we denote
\begin{equation*}
    \mathcal{M}_{\mathrm{sc}} \big( ( \Lap_{\gamma, \theta_{n}} )_{n \in \N} \big)
\end{equation*}
as the set of all possible weak-$\star$ accumulation points of sequences of Wigner distributions $(W_{h_n}[u_n])_{n \in \N}$ where $u_n \in L^2(\S[2])$ and $h_n > 0$ satisfy
\begin{equation} \label{Eq:semiclassical-un_hn-conditions}
    (h_n)^2\Lap_{\gamma, \theta_n} u_n = u_n \ , \qquad \norm{u_n}[L^2(\S[2])] = 1 \ , \qquad \lim_{n \to \infty} h_n = 0 \ .
\end{equation}

\begin{Theorem} \label{Thm:equal-set_SDM}
    Let $\gamma$ be a closed geodesic on $\S[2]$, and let $(\Lap_{\gamma, \theta_n})_{n \in \N}$ be a sequence of singular perturbations of $\Lap$ by $\delta_{\gamma}$. Then
    \[
    \mathcal{P}_{\mathrm{inv}} (S^*\S[2]) \subseteq \mathcal{M}_{\mathrm{sc}} \big( ( \Lap_{\gamma, \theta_{n}} )_{n \in \N} \big) \ .
    \]
    If one further assumes that for every $n \in \N$ (see \eqref{Eq:set-newEV}),
    \begin{equation} \label{Eq:condition-newEV_cap_oldEV=empty}
        \Lambda_{\theta_n} \cap \Specodd(\Lap) = \varnothing \ ,
    \end{equation}
    then
    \[
    \mathcal{P}_{\mathrm{inv}} (S^*\S[2]) = \mathcal{M}_{\mathrm{sc}} \big( ( \Lap_{\gamma, \theta_{n}} )_{n \in \N} \big) \ .
    \]
\end{Theorem}

\begin{Remark} \label{Rmk:genericity_condition-newEV_cap_oldEV=empty}
    Condition
    \[
    \Lambda_{\theta} \cap \Specodd(\Lap) = \varnothing
    \]
    is satisfied for all but a countable collection of $\theta \in \Theta$. This is because $\Specodd(\Lap)$ is countable and if $\lambda_{2\ell + 1}^{2} \in \Lambda_{\theta}$, then $\lambda_{2\ell + 1}^{2} \notin \Lambda_{\theta'}$ for $\theta' \neq \theta$ (see Lemma \ref{Lemma:EV-equation}).
\end{Remark}

\begin{proof}
    \textbf{A.} We prove
    \[
    \mathcal{P}_{\mathrm{inv}}(S^*\S[2]) \subseteq \mathcal{M}_{\mathrm{sc}} \big( ( \Lap_{\gamma, \theta_{n}} )_{n \in \N} \big) \ .
    \]
    Let $\mu \in \mathcal{P}_{\mathrm{inv}}(S^*\S[2])$. Take $q \in \S[2]$ such that $d(q, \gamma) = \frac{\pi}{2}$, and let $(\ell_n)_{n \in \N} \subseteq \N$ and $(u_n)_{n \in \N} \subseteq L^2(\S[2])$ be given by Theorem 1.3 of \cite{Verdasco2026spheres}, which enjoy the following properties
    \[
    \Lap u_n = \lambda_{\ell_{n}}^2 u_n \ , \qquad \norm{u_n}[L^2(\S[2])] = 1 \ , \qquad u_n(q) = 0 \ ,
    \]
    and additionally, if $h_n \coloneqq \lambda_{\ell_{n}}^{-1}$, then $h_n \to 0^+$ and for all $a \in \CinfK(T^*\S[2])$,
    \[
    \lim_{n \to \infty} \ip{u_n}{ \Op[h_n]{a} u_n}[L^2(\S[2])] = \int_{T^*\S[2]} a \, \mu \ .
    \]
    If we prove $\delta_{\gamma} u_n = 0$ for all $n \in \N$, then $(h_n)_{n \in \N}$ and $(u_n)_{n \in \N}$ would satisfy \eqref{Eq:semiclassical-un_hn-conditions}, thus proving that $\mu \in \mathcal{M}_{\mathrm{sc}}((\Lap_{\gamma, \theta_n})_{n \in \N})$. Fix $n \in \N$. Since $u_n \in \ker(\Lap - \lambda_{\ell_n}^2)$, \eqref{Eq:def-Yellgamma} gives that
    \[
    \delta_{\gamma} u_n = \ip{Y_{\ell_n}^{\gamma}}{u_n} \ .
    \]
    Next, Lemma \ref{Lemma:relation_Yellgamma-Zellq} implies that $Y_{\ell_n}^{\gamma} = \alpha_{n} Z_{\ell_n}^{q}$ for some $\alpha_{n} \in \R$, therefore
    \[
    \delta_{\gamma} u_n = \alpha_{n} \ip{Z_{\ell_n}^{q}}{u_n}[L^2(\S[2])] \ .
    \]
    Lastly, since $u_n \in \ker(\Lap - \lambda_{\ell_n}^2)$, \eqref{Eq:def-Zellq} gives
    \[
    \delta_{\gamma} u_n = \alpha_n u_n(q) = 0 \ . \qedhere
    \]

    \textbf{B.} We prove
    \[
    \mathcal{M}_{\mathrm{sc}} \big( ( \Lap_{\gamma, \theta_{n}} )_{n \in \N} \big) \subseteq \mathcal{P}_{\mathrm{inv}} (S^*\S[2])
    \]
    under the additional hypothesis \eqref{Eq:condition-newEV_cap_oldEV=empty}. Fix $\mu \in \mathcal{M}_{\mathrm{sc}}((\Lap_{\gamma, \theta_n})_{n \in \N})$. Let $(u_n)_{n \in \N} \subseteq L^2(\S[2])$ and $(h_n)_{n \in \N} \subseteq (0,1)$, be sequences satisfying \eqref{Eq:semiclassical-un_hn-conditions} and additionally, for all $a \in \CinfK(T^*\S[2])$,
    \[
    \lim_{n \to \infty} \ip{u_n}{\Op[h_n]{a} u_n}[L^2(\S[2])] = \int_{T^*\S[2]} a \, \mu \ .
    \]
    Since \eqref{Eq:condition-newEV_cap_oldEV=empty} holds for all $n \in \N$, Theorem \ref{Thm:Spectrum-Lapgammatheta} implies that for every $n \in \N$, one of the following two mutually exclusive events occurs:
    \begin{enumerate}
        \item $h_n^{-2} \in \Spec(\Lap)$ and $u_n \in \ker(\Lap - h_n^{-2})$, $\delta_{\gamma} u_n = 0$.
        \item $h_n^{-2} \in \R \setminus \Spec(\Lap)$ and $u_n = c g_{h_n}^{\gamma}$, $c \in \C$, $\abs{c} = 1$.
    \end{enumerate}
    Up to extraction of a subsequence, we may assume that (1) holds for all $n \in \N$, or that (2) holds for all $n \in \N$. If (1) holds for all $n \in \N$, then $\mu \in \mathcal{P}_{\mathrm{inv}}(S^*\S[2])$. If (2) holds for all $n \in \N$, Theorem \ref{Thm:SDM-newEF_is_muq} implies that $\mu = \mu_q \in \mathcal{P}_{\mathrm{inv}}(S^*\S[2])$.
\end{proof}

For $q \in \S[2]$ and a Lebesgue-measurable set $J \subseteq \R / (2 \pi \Z)$, let $\mu_{q, J}$ be the positive Radon measure on $T^*\S[2]$,
\begin{equation} \label{Eq:def-muqJ}
    \int_{T^*\S[2]} a \, \mu_{q, J} \coloneqq \int_{S_{q}^*\S[2]} \int_{J} a(\phi_t(q, \xi)) \frac{\D{t}}{2\pi} \frac{\D{\xi}}{2\pi} \ , \qquad a \in \Cont(T^*\S[2]) \ .
\end{equation}
Note that $\mu_{q, J}$ is a probability measure if and only if $J$ is of full measure in $\R / (2 \pi \Z)$, and if $J$ and $K$ are two disjoint Lebesgue measurable subsets of $\R / (2 \pi \Z)$, then $\mu_{q, J} + \mu_{q, K} = \mu_{q, J \cup K}$.

\begin{Theorem} \label{Thm:end-theorem-non-invariance}
    Let $\gamma$ be a closed geodesic on $\S[2]$, and let $(\Lap_{\gamma, \theta_n})_{n \in \N}$ be a sequence of singular perturbations of $\Lap$ by $\delta_{\gamma}$. Assume that there exists $(\ell_n)_{n \in \N} \subseteq \N$ such that (see \eqref{Eq:set-newEV})
    \begin{equation} \label{Eq:bad-sequence}
        \lambda_{2\ell_n + 1}^{2} \in \Lambda_{\theta_n} \qquad \text{and} \qquad \lim_{n \to \infty} \ell_n = \infty \ .
    \end{equation}
    Set $q \in \S[2]$ such that $d(q, \gamma) = \frac{\pi}{2}$, and intervals $J \coloneqq [\frac{\pi}{2}, \frac{3\pi}{2})$ and $K \coloneqq \R / (2\pi \Z) \setminus J$. Take $m_{J}, m_{K} \in [0,2]$ such that $m_{J} + m_{K} = 2$. Moreover, set $h_n \coloneqq \lambda_{2 \ell_n + 1}^{-1}$. There exist $(u_n)_{n \in \N} \subseteq L^2(\S[2])$ satisfying \eqref{Eq:semiclassical-un_hn-conditions} such that for all $a \in \CinfK(T^*\S[2])$,
    \[
    \lim_{n \to \infty} \ip{ u_n }{ \Op[h_n]{a} u_n }[L^2(\S[2])] = \int_{T^*\S[2]} a \ \Big[ m_{J} \mu_{q, J} + m_{K} \mu_{q, K} \Big] \ .
    \]
\end{Theorem}

\begin{proof}
    Set $\varphi, \psi \in \C$ such that $m_{J} = \abs{\varphi + \psi}^2$ and $m_{K} = \abs{\varphi - \psi}^2$. Due to hypothesis \eqref{Eq:bad-sequence}, Theorem \ref{Thm:Spectrum-Lapgammatheta} implies that
    \[
    u_n \coloneqq \varphi \, z_{2\ell_n + 1}^{q} + \psi \, g_{h_n}^{\gamma} \in \ker(\Lap_{\gamma, \theta_n} - \lambda_{2\ell_n + 1}^{2}) \qquad \forall \, n \in \N \ .
    \]
    Moreover, since $z_{2\ell_n + 1}^{q}$ and $g_{h_n}^{\gamma}$ are mutually orthogonal ($g_{h}^{\gamma} \in \ker(\Lap - \lambda_{2j + 1}^2)^{\perp}$ for all $h > 0$, $j \in \N$), one has
    \[
    \norm{u_n}[L^2(\S[2])]^2 = \abs{\varphi}^2 + \abs{\psi}^2 = \frac{1}{2} \big( \abs{\varphi + \psi}^2 + \abs{\varphi - \psi}^2 \big) = \frac{1}{2} \big( m_{J} + m_{K} \big) = 1 \ ,
    \]
    therefore $(u_n)_{n \in \N} \subseteq L^2(\S[2])$ satisfies \eqref{Eq:semiclassical-un_hn-conditions}. We claim that $(u_n)_{n \in \N}$ is the sequence we were looking for.

    Up to extraction of a subsequence, we may assume there exists a unique measure $\mu$ on $T^*\S[2]$ such that
    \[
    \lim_{n \to \infty} \ip{ u_n }{\Op[h_n]{a} u_n }[L^2(\S[2])] = \int_{T^*\S[2]} a \, \mu \qquad \forall \, a \in \CinfK(T^*\S[2]) \ .
    \]
    Fix $a \in \CinfK(T^*\S[2])$. Note that if for $n \in \N$ we set $\sigma_n \in (0,1)$ such that
    \[
    \lambda_{2\ell_n + 1}^{2} = (h_n)^{-2} = \big( 2(\ell_n + \sigma_n) \big) \big( 2(\ell_n + \sigma_n) + 1 \big) \ ,
    \]
    then $\sigma_n = \frac{1}{2}$ for all $n \in \N$, thus certainly $\sigma_n \to \sigma = \frac{1}{2}$. Let $\Upsilon \in \N$ be a large parameter. Using the same reasoning as in the proof of Theorem \ref{Thm:SDM-newEF_is_muq}, part B, we infer that there exist a smooth symbol $\Gamma^{\frac{1}{2}, \Upsilon}$ \eqref{Eq:def-GammasigmaUpsilon}, constants $C, C(\Upsilon) > 0$, and $N = N(\Upsilon) \in \N$ such that for all $n \geq N$ (see \eqref{Eq:ghgamma-approx-PsDO(symbol)zelln}),
    \[
    \norm*{ g_{h_n}^{\gamma} - \Op[h_n]{\Gamma^{\frac{1}{2}, \Upsilon} } z_{2\ell_n + 1}^{q} }[L^2(\S[2])] \leq C \Big[ C(\Upsilon) h_n + \Upsilon^{-\frac{1}{2}} \Big] \ .
    \]
    Consequently, we find that for all $\Upsilon \in \N$ and all $n \geq N(\Upsilon)$,
    \[
    \norm*{ u_n - \Op[h_n]{\varphi + \psi \, \Gamma^{\frac{1}{2}, \Upsilon} } z_{2\ell_n + 1}^{q} }[L^2(\S[2])] \leq C \Big[ C(\Upsilon) h_n + \Upsilon^{-\frac{1}{2}} \Big] \ .
    \]
    Using symbolic calculus, we get that for all $\Upsilon \in \N$ and all $n \geq N(\Upsilon)$,
    \begin{equation*} 
        \abs*{ \ip{ u_n }{\Op[h_n]{a} u_n }[L^2(\S[2])] - \ip{ z_{2 \ell_n + 1}^{q} }{\Op[h_n]{a \abs*{\varphi + \psi \, \Gamma^{\frac{1}{2}, \Upsilon}}^2 } z_{2\ell_n + 1}^{q} }[L^2(\S[2])] } \leq C \Big[ C(\Upsilon) h_n + \Upsilon^{-\frac{1}{2}} \Big] \ .
    \end{equation*}
    Thanks to Theorem 3.1 from \cite{Verdasco2026spheres}, we may take limit $n \to \infty$ and obtain for all $\Upsilon \in \N$,
    \begin{equation} \label{Eq:auxend1123}
        \abs*{ \int_{T^*\S[2]} a \, \mu - \int_{T^*\S[2]} a \abs*{\varphi + \psi \, \Gamma^{\frac{1}{2}, \Upsilon}}^2 \, \mu_q } \leq C \Upsilon^{-\frac{1}{2}} \ .
    \end{equation}
    We would like to take $\Upsilon \to \infty$ now, but first we need to compute the weak-$\star$ limit of the sequence of measures $( \abs*{\varphi + \psi \, \Gamma^{1/2, \Upsilon}}^2 \, \mu_q )_{\Upsilon \in \N}$. We see that
    \[
    \int_{T^*\S[2]} a \abs*{\varphi + \psi \Gamma^{\frac{1}{2}, \Upsilon}}^2 \, \mu_q = \int_{S_{q}^{*}\S[2]} \int_{0}^{2\pi} a(\phi_t(q, \xi)) \abs*{\varphi + \psi \, \Gamma^{\frac{1}{2}, \Upsilon} (\phi_t(q, \xi))}^2 \frac{\D{t}}{2\pi} \frac{\D{\xi}}{ \vol(S_{q}^{*}\S[2]) } \ .
    \]
    Due to the definition of $\Gamma^{\frac{1}{2}, \Upsilon}$ \eqref{Eq:def-GammasigmaUpsilon} and the properties \eqref{Eq:kappaq-varsigmaq-flowoutq},
    \begin{equation} \label{Eq:auxend34273}
        \varphi + \psi \, \Gamma^{\frac{1}{2}, \Upsilon} (\phi_t(q, \xi)) = \varphi + \psi \frac{e^{-it}}{C_{\frac{1}{2}}}\sum_{k = -\Upsilon + 1}^{\Upsilon} \frac{(-1)^k}{k - \frac{1}{2}} e^{i k t} \ .
    \end{equation}
    Therefore, Lemma \ref{Lemma:zetasigma2-Fourierexp} for $\sigma = \frac{1}{2}$ implies that \eqref{Eq:auxend34273} is the partial Fourier series of
    \[
    \Big[ \varphi + \psi \, e^{-it} \zeta^{\frac{1}{2}, 2} \Big] (t) = \varphi + \psi \big[ \charf_{J}(t) - \charf_{K}(t) \big] \ ,
    \]
    where $J \coloneqq [\frac{\pi}{2}, \frac{3\pi}{2})$ and $K \coloneqq \R / (2 \pi \Z) \setminus J$. Carleson's theorem gives that for all $\xi \in S_{q}^{*}\S[2]$
    \[
    \abs*{ \varphi + \psi \, \Gamma^{\frac{1}{2}, \Upsilon} (\phi_t(q, \xi)) }^2 \xrightarrow[\Upsilon \to \infty]{} \abs{\varphi + \psi}^2 \charf_{J}(t) + \abs{\varphi - \psi}^2 \charf_{K}(t) \qquad \text{a.e. $t \in \R / (2\pi \Z)$.}
    \]
    The dominated convergence theorem let us take $\Upsilon \to \infty$ in \eqref{Eq:auxend1123} and conclude
    \[
    \int_{T^*\S[2]} a \, \mu = \int_{T^*\S[2]} a \ \Big[ m_{J} \mu_{q, J} + m_{K} \mu_{q, K} \Big] \ . \qedhere
    \]
\end{proof}

\begin{Lemma} \label{Lemma:projection_into_QL}
    Let $\gamma$ be a closed geodesic in $\S[2]$ and let $q \in \S[2]$ be such that $d(q, \gamma) = \frac{\pi}{2}$. Consider the following smooth map
    \[
    f \colon \R / (2\pi \Z)  \times S_{q}^{*}\S[2] \to \S[2] \ , \qquad f(t, \xi) \coloneqq \exp_{q}(t\xi) \ ,
    \]
    and the measures $\mu$ on $\R / (2\pi \Z)  \times S_{q}^{*}\S[2]$ and $\nu_{\gamma}$ on $\S[2]$,
    \begin{align*}
        \mu(\D{t}, \D{\xi}) & \coloneqq \frac{\D{t}}{2 \pi} \frac{\D{\xi}}{\vol(S_{q}^{*}\S[2])} \\
        \nu_{\gamma}(\D{x}) & \coloneqq \frac{2}{\pi} \frac{1}{\cos(d(x, \gamma))} \frac{\D{x}}{ \vol(\S[2])} = \frac{2}{\pi} \frac{1}{\sin(d(x, q))} \frac{\D{x}}{ \vol(\S[2])}
    \end{align*}
    Assume that $U$ is an open connected subset of $\R / (2\pi \Z)  \times S_{q}^{*}\S[2]$ such that $U \subseteq \{ (t, \xi) \colon t \neq 0, \pi\}$. Then $f|_{U}$ is a diffeomorphism onto its image and the pushforward of $\mu$ by $f|_{U}$ is
    \[
    (f|_{U})_{*} \mu = \frac{1}{2} \nu_{\gamma} \ .
    \]
\end{Lemma}

\appendix

\section{The von Neumann theory of symmetric extensions}
\label{App:vonNeumann_theory}

In this appendix, we collect and prove some abstract results on the von Neumann theory of symmetric extensions that we invoked in the article.

Let $\mathcal{H}$ be a complex Hilbert space with inner product $\ip{u}{v}$, $u, v \in \mathcal{H}$. We will assume that it is conjugate-linear in the first argument:
\[
\ip{z u + w}{v} = \conj{z} \ip{u}{v} + \ip{w}{v} \ .
\]
Let $T \colon \dom(T) \to \mathcal{H}$ be a closed symmetric operator, where $\dom(T)$ is a dense subspace of $\mathcal{H}$.  

\begin{Definition}
    The adjoint operator $T^* \colon \dom(T^*) \to \mathcal{H}$ is the operator given by
    \begin{equation*}
        \dom(T^*) \coloneqq \Big\{ u \in \mathcal{H} \colon \ \exists \, C_{u} > 0 \ \text{s.t.} \ \abs{\ip{u}{T w}} \leq C_{u} \norm{w} \quad \forall \, w \in \dom(T) \, \Big\} \ ,
    \end{equation*}
    and for $u \in \dom(T^*)$, $T^*u \in \mathcal{H}$ is defined via the Riesz representation theorem: $T^*u$ is the unique vector such that
    \[
    \ip{T^*u}{w} = \ip{u}{T w} \qquad \forall \, w \in \dom(T) \ .
    \]
\end{Definition}

Then, we consider two additional structures on $\dom(T^*)$. First, we define the graph--inner product on $\dom(T^*)$:
\[
\ip{u}{w}[T^*] \coloneqq \ip{u}{w} + \ip{T^* u}{T^* w} \ , \qquad u, w \in \dom(T^*) \ .
\]
Since $T^*$ is closed, $(\dom(T^*)$ is a Hilbert space with this inner product, and since $T$ is closed, $\dom(T)$ is a closed subspace of $\dom(T^*)$. Additionally, we also define a skew-Hermitian form on $\dom(T^*)$:
\begin{equation} \label{Eq:def-omega_form}
    \omega(u, v) \coloneqq \ip{u}{T^*v} - \ip{T^*u}{v} \ , \qquad u, v \in \dom(T^*) \ .
\end{equation}

\begin{Lemma} \label{Lemma:characterization_dom(T)-via_omega}
    Given a closed symmetric operator $T$ on $\mathcal{H}$ and the skew-Hermitian form $\omega$ on $\dom(T^*)$, one has
    \[
    u \in \dom(T) \quad \iff \quad \omega(u, w) = 0 \quad \forall \, w \in \dom(T^*) \ .
    \]
\end{Lemma}

\begin{proof}
    If $u \in \dom(T)$, then $T^*u = Tu$ because $T \subset T^*$. Therefore, for every $w \in \dom(T^*)$, one has
    \[
    \omega(u, w) = \ip{u}{T^*w} - \ip{T^*u}{w} = \ip{Tu}{w} - \ip{Tu}{w} = 0 \ .
    \]

    If $u \in \dom(T^*)$ is such that $\omega(u, w) = 0$ for all $w \in \dom(T^*)$, then $u \in \dom((T^{*})^{*}) = \dom(\cl{T}) = \dom(T)$ because $T$ is closed. This is a direct computation: for every $w \in \dom(T^*)$,
    \[
    \ip{u}{T^*w} = \omega(u, w) + \ip{T^*u}{w} = \ip{T^*u}{w} \ . \qedhere
    \]
\end{proof}

\begin{Definition} \label{Def:omega-orthogonal-subspace}
    For a given subspace $L \subseteq \dom(T^*)$, we define the $\omega$-orthogonal subspace of $L$ as follows
    \[
    L^\omega \coloneqq \big\{ u \in \dom(T^*) \colon \omega(u, w) = 0 \ \forall \, w \in L \} \ .
    \]
    Thus, we say that a linear subspace $L$ is isotropic if and only if $\dom(T) \subseteq L \subseteq L^{\omega}$ \footnote{Lemma \ref{Lemma:characterization_dom(T)-via_omega} implies that $\dom(T) \subseteq L^\omega$ for any linear subspace $L$.}, and $L$ is Lagrangian if and only if $L = L^\omega$.
\end{Definition}

\begin{Prop} \label{Prop:self-adjoint-characterization}
    Let $T$ be a closed symmetric operator, and let $\omega$ be the skew-Hermitian form \eqref{Eq:def-omega_form} on $\dom(T^*)$. For a subspace $\dom(T) \subseteq L \subseteq \dom(T^*)$, define the operator $T_{L} \coloneqq T^*|_{L}$.
    \begin{enumerate}
        \item $(T_{L})^* = T_{L^{\omega}}$
        \item $T_{L}$ is symmetric if and only if $L$ is isotropic.
        \item $T_{L}$ is self-adjoint if and only if $L$ is Lagrangian.
    \end{enumerate}
\end{Prop}

\begin{proof}
    Let $L$ be a subspace as described above. To prove item (1), it suffices to show that $\dom( (T_{L})^* ) = L^{\omega}$. If that were the case, for all $u \in L^{\omega}$ and all $w \in L$,
    \[
    \ip{ (T_{L})^* u}{w} = \ip{u}{T_{L} w} = \ip{u}{T^* w} = \omega(u, w) + \ip{T^*u}{w} = \ip{T^*u}{w} \ .
    \]
    
    Assume first that $u \in \dom( (T_{L})^*)$. We are going to show that $\omega(u, w) = 0$ for all $w \in L$. First, by the definition of $\dom( (T_{L})^* )$, there exists $C_{u} > 0$ such that
    \[
    \abs*{ \ip{u}{T^*w} } \leq C_{u} \norm{w} \qquad \forall \, w \in L \ .
    \]
    Since $\dom(T) \subseteq L$, we infer that $u \in \dom(T^*)$. This implies that
    \[
    \omega_{u}(w) \coloneqq \omega(u, w) = \ip{u}{T^*w} - \ip{T^*u}{w} \ , \qquad w \in \dom(T^*) \ ,
    \]
    is a bounded functional over $L$ in the $\mathcal{H}$ norm. Since $\omega_{u}(w) = 0$ for all $w \in \dom(T)$ and $\dom(T)$ is a dense subspace of $\mathcal{H}$, continuity implies that $\omega_{u} \equiv 0$, thus $u \in L^{\omega}$.

    Assume now that $u \in L^{\omega}$. Therefore, for all $w \in L$,
    \[
    \ip{u}{T_{L} w} = \ip{u}{T^* w} = \omega(u, w) + \ip{T^*u}{w} = \ip{T^* u}{w} \ ,
    \]
    which implies that $u \in \dom( (T_{L})^* )$, and $(T_{L})^* u = T^*u$ in addition.

    Items (2) and (3) are direct consequences of item (1).
\end{proof}

In fact, the inner product structure and the skew-Hermitian structure are related to each other through the operator $T^*$.

\begin{Lemma} \label{Lemma:def-J}
    Let $T$ be a closed symmetric operator on $\mathcal{H}$ and let $E$ be the orthogonal complement of $\dom(T)$ inside $\dom(T^*)$. Define the operator $J \coloneqq T^*|_{E}$.
    \begin{enumerate}
        \item For all $u \in E$, $T^*u \in \dom(T^*)$ and $T^*(T^* u) = -u$.
        \item If $u \in E$, then $\omega(u, w) = \ip{T^*u}{w}[T^*]$ for all $w \in \dom(T^*)$.
        \item Assume that $A$ is a self-adjoint extension of $T$ and that $\dim( \dom(A) / \dom(T) ) < \infty$. Then
        \begin{equation} \label{Eq:identity-dimensions-T^*-T}
            \dim( \dom(T^*) / \dom(T) ) = 2 \dim( \dom(A) /\dom(T) ) \ .
        \end{equation}
    \end{enumerate}
\end{Lemma}

\begin{proof}
    Let $u \in E$, that is, such that $\ip{u}{w}[T^*] = 0$ for all $w \in \dom(T)$. To show that $Ju = T^* u \in E$ we study $\ip{T^*u}{Tw}$ for $w \in \dom(T)$.
    \[
    \ip{T^*u}{Tw} = \ip{T^*u}{T^*w} = \ip{u}{w}[T^*] - \ip{u}{w} = \ip{-u}{w} \ .
    \]
    From this identity, we read $T^*u \in \dom(T^*)$ and $T^*(T^*u) = -u$. Item (2) is a direct computation from item (1).

    Assume that $A$ is a self-adjoint extension of $T$. Since $T$ is symmetric, then $A = A^* \subset T^*$, thus $A = T^*|_{L}$ for some Lagrangian subspace $L$ of $\dom(T^*)$ (Proposition \ref{Prop:self-adjoint-characterization}). Let us write $L = \dom(T) + S$, where $S \subseteq E$. We will prove that
    \begin{equation} \label{Eq:dom(T*)-description}
        \dom(T^*) = \dom(T) + S + T^*[S] \ ,
    \end{equation}
    and that $T^*[S]$ and $S$ are $T^*$-orthogonal; if $\dim( \dom(T^*) / \dom(T) ) < \infty$, using that $T^*|_{E}$ is an isomorphism of $E$, one obtains \eqref{Eq:identity-dimensions-T^*-T}.

    To prove \eqref{Eq:dom(T*)-description} and that $T^*[S]$ and $S$ are $T^*$-orthogonal, it suffices to show
    \[
    L^{\perp} \coloneqq \big\{ u \in \dom(T^*) \colon \ip{u}{w}[T^*] = 0 \ \forall \, w \in L \big\} = J[S] \ .
    \]
    If $u \in L^{\perp}$, then $u \in E$ because $\ip{u}{w}[T^*] = 0$ for all $w \in \dom(T)$. Due to item (1), $T^*u \in E$ too. In addition, for every $w \in L$, due to item (2),
    \[
    0 = \ip{u}{w}[T^*] = \omega(T^*u, w) \ ,
    \]
    thus $T^*u \in L^{\omega} = L$. Since $T^*u \in E \cap L$, we have $T^* u \in S$, therefore $u = T^*[-T^* u] \in J[S]$.

    Now, let $v \in S$ and take $u = T^*v \in E$. We have for every $w \in L$,
    \[
    \ip{u}{w}[T^*] = \ip{T^*v}{w}[T^*] = \omega(v, w) = 0 \ ,
    \]
    because $v \in S \subseteq L = L^{\omega}$, hence $u \in L^{\perp}$.
\end{proof}

\printbibliography

\end{document}

%% file: Repository/packages.tex
\usepackage[utf8]{inputenc}
\usepackage[T1]{fontenc}
\usepackage{geometry}
\usepackage[sorting = nyt, sortcites = true, giveninits = true, isbn = false]{biblatex}                          
\AtEveryBibitem{\clearfield{number}\clearfield{pages}}
\renewbibmacro{in:}{}

\usepackage{amssymb}
\usepackage{mathrsfs} 
\usepackage{mathtools} 
\usepackage{amsthm} 
\usepackage{thmtools} 
\usepackage{dsfont}

\usepackage[dvipsnames]{xcolor}
\usepackage[shortlabels]{enumitem} 
\usepackage[colorlinks = true]{hyperref}

\theoremstyle{plain}
\declaretheorem[parent = section]{Theorem}
\declaretheorem[sibling = Theorem]{Lemma}

\declaretheorem[name = Proposition, sibling = Theorem]{Prop}

\theoremstyle{definition}
\declaretheorem[sibling = Theorem]{Definition}

\declaretheorem[sibling = Theorem]{Remark}

\numberwithin{equation}{section} 



%% file: Repository/commands.tex
\DeclareDocumentCommand{\R}{O{} O{}}{ \mathbb{R}_{#2}^{#1} }
\DeclareDocumentCommand{\C}{O{} O{}}{ \mathbb{C}_{#2}^{#1} }
\DeclareDocumentCommand{\S}{O{} O{}}{ \mathbb{S}_{#2}^{#1} }
\DeclareDocumentCommand{\T}{O{} O{}}{ \mathbb{T}_{#2}^{#1} }
\DeclareDocumentCommand{\Z}{O{} O{}}{ \mathbb{Z}_{#2}^{#1} }
\def \N{\mathbb{N}}

\def \Cont{\mathcal{C}}

\def \Cinf{\mathcal{C}^{\infty}}
\def \CinfK{\mathcal{C}_{c}^{\infty}}
\def \Dist{\mathcal{D}'}

\def \charf{\mathds{1}}

\newcommand \littleo{
  \mathchoice
    {\scalebox{0.9}{$\scriptstyle\mathcal{O}$}}
    {\scalebox{0.9}{$\scriptstyle\mathcal{O}$}}
    {\scalebox{0.9}{$\scriptscriptstyle\mathcal{O}$}}
    {\scalebox{0.7}{$\scriptscriptstyle\mathcal{O}$}}
  }
\def \BigO{\mathcal{O}}

\newcommand \Liouv{
\mathchoice
    {{\scriptstyle\mathcal{L}}}
    {{\scriptstyle\mathcal{L}}}
    {{\scriptscriptstyle\mathcal{L}}}
    {\scalebox{.7}{$\scriptscriptstyle\mathcal{L}$}}
  }

\DeclarePairedDelimiter{\abs}{\lvert}{\rvert}

\NewDocumentCommand{\norm}{s O{} m O{}}
    {
        \IfBooleanTF{#1}
            {\ensuremath{{\left\lVert #3 \right\rVert}_{#4}}}
            {{#2\lVert {#3} #2\rVert}_{#4}}
    } 

\NewDocumentCommand{\ip}{s O{} m m O{}}
    {
        \IfBooleanTF{#1}
            {\ensuremath{{\left\langle #3 \, \middle| \, #4 \right\rangle}_{#5}}}
            {{#2\langle #3 \, #2| \, #4 #2\rangle}_{#5}}
    }

\DeclarePairedDelimiter{\bra}{\langle}{|}
\DeclarePairedDelimiter{\ket}{|}{\rangle}

\newcommand*{\cl}[1]{\mkern 2mu\overline{\mkern-1.5mu#1\mkern-1.5mu}\mkern 2mu} 
\newcommand*{\conj}[1]{ \mkern 2mu \overline{\mkern-1mu #1 \mkern-1mu} \mkern 2mu}
\newcommand*{\D}[1]{\mathop{\mathrm{d}#1}}

\DeclareMathOperator{\spn}{Span}

\NewDocumentCommand{\Span}{O{} O{} m}{\spn_{#1} #2\{ #3 #2\}}

\let\Im\relax

\DeclareMathOperator{\Im}{\mathfrak{Im}}
\DeclareMathOperator{\dom}{dom}

\DeclareMathOperator{\Lap}{\Delta}
\def \sqLap{\sqrt{\Lap}}
\DeclareMathOperator{\supp}{supp} 
\DeclareMathOperator{\Opp}{Op}
\DeclareMathOperator{\vol}{vol}
\DeclareMathOperator{\Spec}{Spec}

\DeclareMathOperator{\cotan}{cotan}

\NewDocumentCommand{\Op}{O{} O{} m}{\ensuremath{\Opp_{#1}^{#2} \left( #3 \right)}}